%% file: thesis_writeup.tex
\documentclass[12pt]{report}

\usepackage[dvips, outer=1in, inner=1in, top=1in, bottom=1.2in]{geometry}

\input{mathdefs.sty}

\usepackage[dvips]{graphicx}
\usepackage{wrapfig}

\newcommand{\I}{\mathrm{I}}

\usepackage[dvips]{graphicx}

\newtheorem{theorem}{Theorem}[section]
\newtheorem{lemma}[theorem]{Lemma}

\newtheorem{prop}[theorem]{Proposition}

\theoremstyle{remark}

\newtheorem{remark}[theorem]{Remark}
\author{Ilya Vinogradov\footnote{Princeton University, Princeton, NJ}}
\date{\today}
\title{Effective bisector estimate with application to Apollonian circle packings}

\begin{document}
\maketitle

\makeatletter

\renewcommand{\chapter}{\@startsection{chapter}{0}{0pt}{3.5ex plus 0.2ex minus 0.1ex}{2.3ex plus 0.1ex minus 0.1ex}{\center\normalfont\sc\large}}

\renewcommand{\section}{\@startsection{section}{1}{0pt}{2ex plus 0.2 ex minus 0.1ex}{-1ex}{\normalfont\bf}}

\makeatother

\renewcommand{\thesection}{\arabic{chapter}.\arabic{section}}

\renewcommand{\thechapter}{\arabic{chapter}.}

\renewcommand{\thetheorem}{\arabic{chapter}.\arabic{section}.\arabic{theorem}}

\renewcommand{\thefigure}{\arabic{chapter}.\arabic{figure}}

\numberwithin{equation}{section}

\newcommand{\If}{i\hspace*{-0.5ex}f\hspace*{0.1ex}}
\newcommand{\Ih}{i\hspace*{-0.2ex}h\hspace*{0.1ex}}
\newcommand{\Ie}{i\hspace*{-0.2ex}e\hspace*{0.1ex}}

\begin{abstract}
\input{abstract}
\end{abstract}

\tableofcontents

\chapter{Introduction}

\input{intro}

\section{Acknowledgements.} \input{acknowledgements}

\input{statement_of_results}

\input{parametrization}

\input{line_model}

\input{automorphic_model}

\input{decay}

\input{proof_of_main_th}

\input{proof_of_th_apollonian}

\bibliographystyle{plain}

\parindent=0pt
\parskip=0pt
\renewcommand{\baselinestretch}{0.95}
\small
%\bibliography{/home/ilya/TeX_files/bibliography}
\bibliography{bibliography}

\bigskip

\normalsize

{\sc
Mathematics Department, Princeton University, Princeton, NJ}

\textit{Email:} \texttt{ivinogra@math.princeton.edu}
\end{document}

%% file: abstract.tex
%We prove an effective bisector estimate for a geometrically finite non-elementary discrete group $\Gamma<\PSL(2,\C)$ with critical exponent greater than 1. The proof uses spectral theory of $L^2(\Gamma\quot G)$. We apply this result to give power savings in the Apollonian circle packing problem and related counting problems. 
Let $\Gamma<\PSL(2,\C)$ be a geometrically finite non-elementary discrete subgroup, and let its critical exponent $\delta$ be greater than 1. We use representation theory of $\PSL(2,\C)$ to prove an effective bisector counting theorem for $\Gamma$, which allows counting the number of points of $\Gamma$ in general expanding regions in $\PSL(2,\C)$ and provides an explicit error term. We apply this theorem to give power savings in the Apollonian circle packing problem and related counting problems. 

%% file: intro.tex
\section{Lattice point counting problem.} Let $G$ be a topological group with a norm $\|\cdot\|$ and let $\Gamma$ be a discrete subgroup. The question of estimating $$N(R)=\#\{\gamma\in\Gamma\colon\|\gamma\|<R\}$$ as $R$ grows is known as the \textit{lattice point counting problem} and was first asked by Gauss for $G=\R^2$, $\|\cdot\|=\|\cdot\|_{L^2}$, and $\Gamma=\Z^2$; he showed that $$N(R)=\pi R^2+E(r)$$ with $|E(r)|<2\sqrt 2 \pi r$. The problem of improving the bound on the error function $E(r)$ is now called  the \textit{Gauss circle problem}. The current record $E(r)=O(r^{\sigma})$ for $\sigma=\frac{131}{208}$ (ignoring logarithmic factors) is  held by Huxley \cite{huxley_integer_points, huxley_exponential3}, who improved previous results by van der Corput \cite{vdcorput_zum} ($\sigma=\frac{27}{41}$), Vinogradov ($\sigma=\frac{34}{53}$), and  others. 

Delsarte \cite{delsarte_gitter_1942} and Huber \cite{huber_neue_1956, huber_analytischen_1959} initiated the study of this question in the hyperbolic setting: $G=\SL(2,\R)$, $\Gamma$ a cocompact lattice, and the norm comes from the distance in $\H^2$. The difficulty that naturally arises here is that due to hyperbolic expansion: the measure of the boundary of the expanding ball grows roughly at the same rate as the measure of the ball itself. For this reason it is not clear at the first sight that the main term for $N(R)$ is given by the volume of the ball of radius $R$. Selberg \cite{selberg_harmonic_1956} was able to produce an excellent error term in the function $N(R)$ for any lattice $\Gamma$ using his celebrated trace formula.

The question of understanding the growth of $N(R)$ for infinite covolume subgroups $\Gamma$ arose naturally. Patterson  \cite{patterson_limit_set} and Sullivan \cite{sullivan_density, sullivan_entropy} developed extremely useful machinery for analyzing infinite covolume geometrically finite groups $\Gamma$ for $G=\SL(2,\R)$ and more generally for $G=\SO(n,1).$ To each such $\Gamma$ they associated a $\Gamma$-invariant probability measure on the boundary of the hyperbolic plane ($n$-space) supported on the limit set of $\Gamma$, and related it to the spectrum of the Laplacian on the manifold $\Gamma\quot \H$ when the base eigenvalue is at least $\left(\frac{n-1}2\right)^2$. Lax and Phillips \cite{lax_phillips} used wave equation techniques to produce a very good error term for the counting problem in $\H^n$ for general $n$ and showed that the spectrum of the Laplacian on $\Gamma\quot \H^n$ has but finitely many non-tempered eigenvalues. The main order of growth in this case is related to the base eigenvalue $ \delta(n-1-\delta)$ of the Laplacian on $\Gamma\quot \H^n$ for some $\delta=\delta_\Gamma\in(\frac{n-1}{2},n-1)$, and equals $\const\cdot e^{\delta R}$ for the number of points in a ball of radius $R$. This is consistent with the fact that the volume of such ball grows like $e^{(n-1)R}$. 

\section{Sector and bisector count.} A question which arises in applications of lattice point counting problem is counting in more general expanding sets, not just in balls. Both in the finite and infinite covolume case this problem can be approached by cutting a ball into sectors and using them to approximate a more general expanding set. In this setting the word \textit{bisector} refers a subset $S$ of a semisimple Lie group $G=KA^+K$ (known as Cartan or polar decomposition) for which there exist contractible $S_{K_1}\subset K$, $S_A\subset A^+$, $S_{K_2}\subset K$ such that $S=S_{K_1}S_A S_{K_2}.$ When $S_{K_1}=K=S_{K_2}$, the set $S$ is nothing but a ball. When one of $S_{K_1}$, $S_{K_2}$ is all of $K$, the set $S$ is referred to as a \textit{sector}. In the case $G=\SO(n,1)$ and $S_{K_2}=K$, the set $S_{K_1}A^+\subset G$ can be viewed as a subset of $\H^n=G/K$, and in the ball model this set will be a hyperbolic (and in fact Euclidean) sector, motivating these definitions. 

%\begin{wrapfigure}[17]{r}{0.39\textwidth}\includegraphics[width=0.37\textwidth]{Apollonian_packing.eps}\caption{Bounded packing\label{fig:boundedpacking}}\end{wrapfigure}
In the case of lattices in $\SL(2,\R)$ a very strong result was obtained by Good \cite{good_local_1983}, and under additional assumptions of finite geometric property and critical exponent $\delta_\Gamma>\frac12$ Bourgain, Kontorovich, and Sarnak \cite{bourgain_sector_2010} proved a similar counting statement without assuming that $\Gamma$ is a lattice. A very general asymptotic bisector counting theorem for lattices was proven by Gorodnik and Nevo \cite{gorodnik_nevo_ergodic_2010}.   Oh and Shah \cite{oh_equidistribution_2010} proved an asymptotic bisector count for $G=\SO(n,1)$ for non-lattices, but their proof relies on measure rigidity and is not readily made effective. In the present paper we prove an effective bisector counting theorem for $G=\SO(3,1)$ for non-lattices (under additional assumptions). Our approach is most similar to that of \cite{bourgain_sector_2010}: the main ingredient is representation theory of $L^2(\Gamma\quot G)$. We develop explicit $K$-type decomposition for complementary series representations of $\PSL(2,\C)$ and then use the adjoint action of $G$ on conjugates of $\Gamma$ to avoid some  computation. 

The difficulty that arises in generalizing the current approach to $\SO(n,1)$ is that parts of the argument are deeply rooted --- albeit less so than that of \cite{bourgain_sector_2010} --- in the structure of $K$-types for complementary series representations of $\SO(3,1)$. Many computations use specific forms of eigenfunctions, which is the main obstacle preventing direct generalization to higher dimensions. Nevertheless we anticipate that an effective bisector estimate should hold for $\SO(n,1)$ for all $n\ge 4$. 

There is another method that can be used to prove bisector theorems (and related limit theorems) in the case when $\Gamma\subset \SO(n,1)$ is convex cocompact. In this case the problem can translated into the language of symbolic dynamics (cf. \cite{lalley_renewal_89, sharp_sector}); we do not pursue this method here. 

\section{Counting $\Gamma$-orbits in cones and hyperboloids.}
Let $Q$ be a form defined over $\Q$. The problem of counting integer vectors $x\in\Z^n$ such that $Q(x)=c\in\Z$ for a given $c$ is very well studied. For definite diagonal forms it is related to Waring's problem \cite{vaughan_wooley_waring}. For a general definite form one can ask for the growth rate of the number of solutions to $Q(x)=c$ as $c\to\infty$. In the case of an indefinite form one can fix $c$ such that $Q(x)=c$ has a solution, and ask for the growth rate of $$\#\{\|x\|<X\mid Q(x)=c\}$$ as $X\to\infty$ for some norm $\|\cdot\|$ on $\Z^n$. When the number of variables $n$ is sufficiently large compared to the degree of the form then the problem may be solved by the Hardy-Littlewood circle method \cite{vaughan_hardy_littlewood} for certain forms $Q$. Using different techniques Duke, Rudnick, and Sarnak \cite{duke_density_1993} and Eskin and McMullen \cite{eskin_mixing_1993} proved that the number of solutions $\#\{\|x\|<X\mid Q(x)=c\}$ grows like the volume of the corresponding ball (under additional assumptions on $Q$).

\begin{figure}[p]\begin{centering}\includegraphics[width=0.4\textwidth]{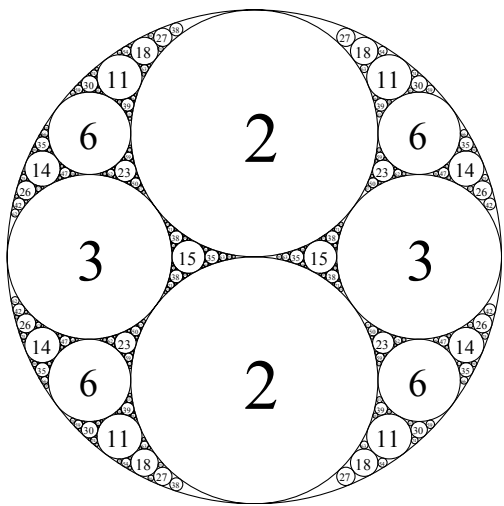}\caption{Bounded packing\label{fig:boundedpacking}}\end{centering}\end{figure}
\begin{figure}[p]\begin{centering}\includegraphics[width=0.75\textwidth]{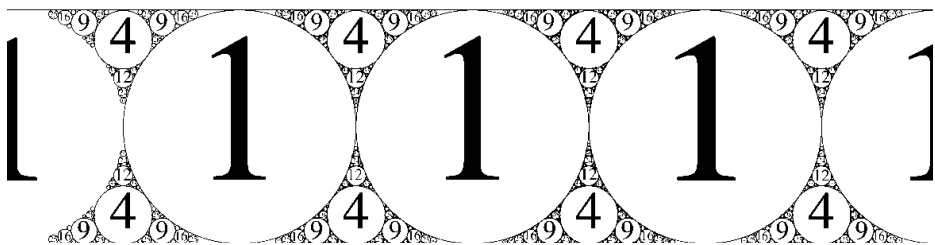}\caption{Unbounded packing\label{fig:unboundedpacking}}\end{centering}\end{figure}
\begin{figure}[p]\begin{centering}\includegraphics[width=0.5\textwidth]{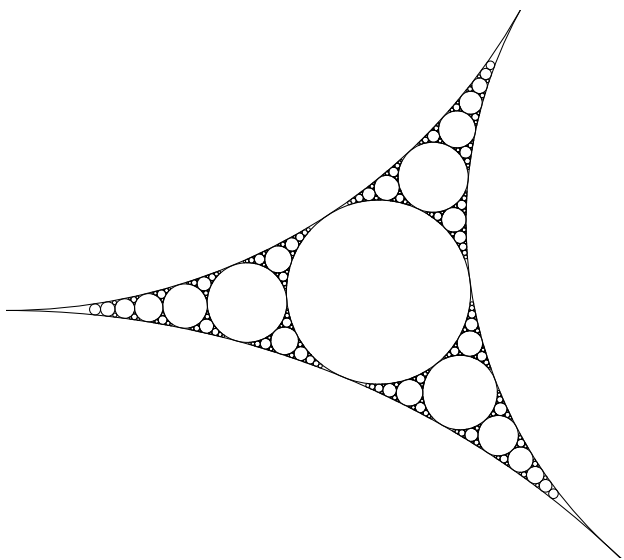}\caption{\label{fig:apollonian_packing_ideal}Apollonian packing in an ideal triangle.}\end{centering}\end{figure}

%\begin{wrapfigure}[8]{r}{0.52\textwidth}\includegraphics[width=0.5\textwidth]{Apollonian_packing1.eps}\caption{Unbounded packing\label{fig:unboundedpacking}}\end{wrapfigure}
In a similar fashion one can count solutions to $\#\{\|x\|<X\mid Q(x)=c\}$ with the additional constraint that $x$ lies in a prescribed orbit of a lattice $\Gamma<\O_Q(\Z)$, but it is a harder question when $\Gamma$ is not a lattice. The bisector counting theorem allows us to compute $\#\{x\in\Gamma x_0 \colon \|x\|<X\}$ with an explicit error term for $G=\SO(3,1)$ and certain groups non-lattices $\Gamma$. Note that the varieties $\{Q=c\}$ are fundamentally different according as $c$ is positive (one-sheeted hyperboloid), negative (two-sheeted hyperboloid), or zero (cone), but our approach works in every case.

A particular instance of this question is the Apollonian circle packing problem. Figures \ref{fig:boundedpacking} and \ref{fig:unboundedpacking} show Apollonian circle packings. For a bounded packing $P$ let $N^P(T)$ be the number of circles in the packing having curvature at most $T$. For a periodic packing $N^P(T)$ is the number of such circles in one period. It has been shown \cite{lagarias1,  kontorovich_apollonian_2011} that the $N^P(T)$ is related to counting points of the orbit of $\Gamma$ in the cone on the so-called Descartes form $Q_D$ which has signature $(3,1)$. The group $\Gamma$ in this case is the Apollonian group (up to finite index) which has infinite index in $\O_{Q_D}(\Z)$. We show that $$N^P(T)=c_P T^\delta+ O(T^{\delta-\eps})$$ for constants $c_P$, $\delta$, and $\eps$. The power savings depend on the spectral gap for $\Gamma$, which is not known numerically. We prove the same counting statement about the packing in Figure \ref{fig:apollonian_packing_ideal}, which is known as the Apollonian gasket or the packing in an ideal triangle. Moreover we can prove effective versions for any circle packing theorem (not just the Apollonian one), such as the ones in \cite{oh_asymptotic_2012} and \cite{oh_counting_2010}.

%% file: acknowledgements.tex
I would like to acknowledge the help, inspiration, patience, and cheerful optimism of Alex Kontorovich, who introduced me to hyperbolic lattice point counting problems,  guided me in every respect, kindly allowed me to use his pictures, and supported my travel through NSF Grant DMS--120937. I am greatly indebted to Yakov G. Sinai for acquainting me with dynamical problems in number theory as well as for his kindness and support. I greatly appreciate discussions with Peter Sarnak who was always supportive and persistent. I would also like thank Stephen Miller for supplying me with vital references.

I am very greatful to Ali Altu\v g, Francesco Cellarosi, and Samuel Ruth for sharing the office with me and being always eager to answer my questions. I also thank Arul Shankar, Jonathan Luk, Mohammad Farajzadeh Tehrani, Kevin Hughes, P\'eter Varj\'u, and other graduate students who made my time at Princeton mathematically fruitful.

This manuscript would have been considerably less complete if it had not been scrutinized by  Hee Oh and Nicolas Templier, whom I thank most kindly. My degree would have never been completed without the help of Jill LeClair, who makes simple that which is intricate. 

Finally, I thank my wife, Milena Zhivotovskaya, and my family for their love and support.

%% file: statement_of_results.tex
\chapter{Statement of results}

\section{Bisector count.} Let $\Gamma <\PSL(2,\C)=G$ be a non-elementary geometrically finite discrete subgroup. 
%A group is called \textit{non-elementary} if it does not contain a finite index abelian subgroup. 
The problem of counting the number of points in a $\Gamma$-orbit which lie in an expanding region in $G$ is well-studied in the case when $\Gamma\quot G$ has finite volume \cite{selberg_harmonic_1956, good_local_1983, duke_density_1993, eskin_mixing_1993}. 
The main term of such a count is always the volume of the expanding region. There is ample literature in the infinite volume case \cite{lalley_renewal_89, sharp_sector, kontorovich_hyperbolic_2009, kontorovich_apollonian_2011, oh_equidistribution_2010}. A major result in this setting is due to Lax and Phillips  \cite{lax_phillips}, who give an unsurpassed error term using a non-Euclidean wave equation. In the present paper we prove a more general counting theorem in the infinite volume case. 

Assume from now on that $\Gamma$ is no longer a lattice. The group $G$ naturally acts on $\H^3=\{x_1+ix_2+jy\mid x_1,x_2\in\R,y>0\}$ by M\"obius transformations (isometries preserving the hyperbolic distance $d$) with quaternion multiplication, and the orbit $\Gamma j\subset \H^3$ is a discrete set. The set of its limit points $\Lambda_\Gamma$ in $\d \H^3$ is known to be a Cantor set whose Hausdorff dimension we denote by $\delta.$ As $\Gamma$ has infinite covolume we have $0<\delta<2$ \cite{beardon_limit_1974}. The Patterson-Sullivan measure is a finite measure $\nu_\Gamma$ supported on $\Lambda_\Gamma$ which is a Hausdorff measure of dimension $\delta$. This measure is constructed (in essence) as the unique weak-$\ast$ limit of the family of measures $$\nu_{s,\Gamma}(z)=\frac{\displaystyle\sum_{\gamma\in\Gamma}e^{-s d(j,\gamma j)}\delta_{\gamma j}(z)}{\displaystyle\sum_{\gamma\in\Gamma}e^{-s d(j,\gamma j)}}$$ as $s\to\delta$ from the right \cite{patterson_limit_set, sullivan_entropy, sullivan_density}. It is the unique finite measure (up to scalar multiples) with the property that $$\frac{d\gamma_\ast\nu_\Gamma}{d\nu_\Gamma}(\xi)=e^{-\delta \beta_\xi(\gamma j,j)}$$ where $\beta_\xi(z_1,z_2)=\lim_{z\to\xi} d(z_1,z)-d(z_2,z)$ is the Busemann function and $\gamma_\ast\nu_\Gamma=\nu_\Gamma\circ \gamma^{-1}$ is the push-forward. 

The group $G$ acts on $L^2(\Gamma\quot G)$: for $g\in G$ and $f\in L^2(\Gamma\quot G)$ we have $$g.f(\Gamma x)=f(\Gamma xg),$$ which is the right regular representation. 
Suppose $\delta>1$; then the hyperbolic Laplacian acts on the smooth functions inside this space and has an eigenfunction with eigenvalue $\lambda_0=\delta(2-\delta)<1$. It has but finitely many eigenvalues $\lambda_n\in(0,1)$ with $0\le n\le d$, as was shown in \cite{lax_phillips}; arrange the eigenvalues in increasing order. We can write  $$\lambda_n=s_n(2-s_n)$$ for $1<s_0<s_1\le\dots s_d<2$ and decreasing in $n$. 
As a $G$-representation the right regular representation decomposes as $$V_{s_0}\oplus \dots\oplus V_{s_d}\oplus V_{\text{temp}},$$ where $V_{s_n}$ are isomorphic to complementary series representations with parameter $s_n$, and $s_0=\delta$. 
%When $\delta>1$, the eigenvalues $\lambda_n$ below 1 of the hyperbolic Laplacian on $\Gamma\quot \H^3$ are related to this decomposition via 

In this setting Lax and Phillips \cite{lax_phillips} prove the following counting theorem. 
\begin{theorem}[Lax-Phillips]\label{th:lax}
Assume $\delta>1$. Then $$\#\{\gamma\in\Gamma\colon d(j,\gamma j)<R\}= c_0 e^{\delta R}+c_1 e^{s_1R}+\dots +c_d e^{s_dR}+O(e^{\frac{\delta+1}2 R}R^5).$$
\end{theorem}
The constants $c_n$ are computed explicitly in terms of the corresponding eigenfunctions with eigenvalue $\lambda_n$.

For $g\in G$, let $k_1(g)a(g)k_2(g)$ be a $KA^+K$ decomposition; it is unique up to factors from $M$, which is the normalizer of $A$ in $K$. For the group $G$ in question we have $$K\cong\PSU(2),\quad A^+\cong \R^+,\quad M\cong\SO(2).$$ Thus, $K/M=S^2$ and $M\quot K=S^2$ with different coordinates. Also define $K(g)$ to be $k_1(g)k_2(g)$. Notice that this definition is independent of the location of the $M$ factor. 

There is an orthonormal basis for $L^2(K/M)\cong L^2(\d\H^3)$ given by spherical harmonics $Y_{ab}(\phi,\theta)$, $a\ge 0$, $|b|\le a$; we always assume that every compact group has measure 1. For each fixed $a$ the collection $\{Y_{ab},|b|\le a\}$ constitutes a basis for the ($2a+1$)-dimensional irreducible representation of $\PSU(2).$ We can, and do, think of spherical harmonics as functions on $K$; in this case they are $M$-invariant on the right. More generally consider $$L^2(K)=\sum_{\tau\text{ irrep of }K} (\dim\tau)\cdot \tau.$$ Define generalized spherical harmonics $$Y_{a;bc}(\phi,\theta,\phi_2)=e^{ib\phi} y_{a;bc}(\theta) e^{ic\phi_2}$$ to be an orthonormal basis for $L^2(K)$ such that for fixed $a$ and $c$ with $|c|\le a$ $$\{Y_{a;bc}\colon |b|\le a\}$$ is a basis for the left regular representation of $K$ on $L^2(K)$; the usual spherical harmonics correspond to $c=0$. Writing $$Y_{aa';bb'c}(\phi,\theta,\psi,\theta_2,\phi_2)=Y_{a;bc}(\phi,\theta,0)\overline{Y_{a';b'c}(\psi,\theta_2,\phi_2)^{-1}},$$ we have a basis for $L^2(K/M.K)$ given by $$\{Y_{aa';bb'c}(\phi,\theta,\psi,\theta_2,\phi_2) \colon |b|,|c|\le a; |b'|,|c|\le a'\}.$$ This basis inherits both left and right transformation properties under $K$. We can write $Y_{aa';bb'c}(K(g))$ using our coordinates; whenever $c=0$ we have $$ Y_{aa';bb'c}(K(g))=Y_{ab}(k_1(g))\overline{Y_{a'b'}(k_2^{-1}(g))}.$$

Finally let $|\cdot|$ be a bi-$K$-invariant norm on $G$ normalized as explained in \eqref{eq:norm} below. We prove the following theorem. 

\begin{theorem}\label{th:main}Assume $\delta>1$. Let $$0<\delta(2-\delta)=\lambda_0<\lambda_1\le\dots\le\lambda_d<1$$ be the eigenvalues below 1 of the hyperbolic Laplacian on $\Gamma\quot \H^3$. Then 
%\begin{multline}\sum_{\gamma\in\Gamma, |\gamma|<T}Y_{a'b'c}(k_1(\gamma))\overline{Y_{abc}(k_2^{-1}(\gamma))}=\frac{4\pi^2 \mathbf{1}_{\{c=0\}} T^\delta}{\delta(\delta-1)} \overline{\int_{K/M}Y_{ab}(k)d\nu_\Gamma(k)} \int_{K/M}Y_{a'b'}(k)d\nu_\Gamma(k) +\\+ \mathbf{1}_{\{c=0\}}c_1(a,b,a',b')T^{s_1}+\dots+ \mathbf{1}_{\{c=0\}}c_d(a,b,a',b')T^{s_d}+\\+O\left(T^{\frac{10\delta+1}{11}}(\log T)^{1/11}(a+1)^{15/11}(a'+1)^{15/11}\right).\end{multline} 
\begin{multline}\label{eq:main}\sum_{\gamma\in\Gamma, |\gamma|<T}Y_{a'a;b'bc}(K(\gamma))=\frac{\pi \mathbf{1}_{\{c=0\}} T^{2\delta}}{\delta(\delta-1)} \overline{\int_{K/M}Y_{ab}(k)d\nu_\Gamma(k)} \int_{K/M}Y_{a'b'}(k)d\nu_\Gamma(k) +\\+ \mathbf{1}_{\{c=0\}}c_1(a,b,a',b')T^{2s_1}+\dots+ \mathbf{1}_{\{c=0\}}c_d(a,b,a',b')T^{2s_d}+\\+O\left(T^{2\frac{10\delta+1}{11}}(\log T)^{1/11}(a+1)^{15/11}(a'+1)^{15/11}\right).\end{multline} 
The numbers $s_1$, \dots, $s_d$ are less than $\delta$ and $>1$ and satisfy $s_n(2-s_n)=\lambda_n$. Additionally, \beq\label{eq:bound_on_coeffs}|c_n(a,b,a',b')|\ll ((a+1)(a'+1))^{2-s_n+\frac12}|c_n(0,0,0,0)|,\eeq and the implied constant depend only on $\Gamma.$
\end{theorem}

\begin{remark}The Patterson-Sullivan \cite{patterson_limit_set, sullivan_entropy} measure in the statement is normalized so that the corresponding base eigenfunction of the Casimir operator has $L^2$ norm 1; we explain this  in \eqref{eq:ps_poisson}. \end{remark}

\begin{remark}The special case with trivial spherical harmonics matches the main term Theorem \ref{th:lax}, but the error term in our treatment is worse.\end{remark}

\begin{remark}\label{rem:different_M}
The Theorem implies that there are cancellations when $c$ is non-zero.  The most interesting case, i.e., the case in which we compute the main term, is $c=0$. We can rewrite it more succinctly as \begin{multline}\label{eq:c=0}\sum_{\gamma\in\Gamma, |\gamma|<T}Y_{a'b'}(k_1(\gamma))\overline{Y_{ab}(k_2^{-1}(\gamma))}=\frac{\pi  T^{2\delta}}{\delta(\delta-1)} \overline{\int_{K/M}Y_{ab}(k)d\nu_\Gamma(k)} \int_{K/M}Y_{a'b'}(k)d\nu_\Gamma(k) +\\+ c_1(a,b,a',b')T^{2s_1}+\dots+c_d(a,b,a',b')T^{2s_d}+\\+O\left(T^{2\frac{10\delta+1}{11}}(\log T)^{1/11}(a+1)^{15/11}(a'+1)^{15/11}\right).\end{multline} 
\end{remark}

\begin{remark}\label{rem:trivial}The Theorem is only interesting when $T$ is sufficiently large depending on $a$, $a'$. If the error term dominates the main term a better bound is obtained from Theorem \ref{th:lax} by observing that $$\|Y_{a'a;b'bc}\|_{L^\infty}\ll \sqrt{(a+1)(a'+1)},$$ whence the right hand side of \eqref{eq:main} can be replaced by $$\ll T^{2\delta}\sqrt{(a+1)(a'+1)}.$$
\end{remark}
 
\bigskip

This theorem is a generalization of the work of Bourgain, Kontorovich, and Sarnak in \cite{bourgain_sector_2010} for $\SL(2,\R)$; we generalize parts of their method to $\PSL(2,\C)$. In the special case $c=0$, the main term has been obtained by Oh and Shah \cite{oh_equidistribution_2010}, but their method uses measure rigidity and therefore is not readily made effective. In the present paper we use spectral theory of $L^2(\Gamma\quot G)$, which provides explicit error terms. For the same reason the method of the present paper cannot be applied once $\delta\le 1$ as there are no $L^2$ eigenfunctions on $\Gamma\quot \H^3$, but Oh and Shah give asymptotics in that case, too. 

The proof of the Main Theorem relies on spectral theory, but we do not follow the original treatment of \cite{bourgain_sector_2010}. We do not attempt to compute the leading term for every combination $a,b,a',b',c$ directly --- although it is probably possible, and we give tools adequate for computing the main term in any particular case --- but instead develop a new method using conjugates of $\Gamma$ to obtain the statement for general indices; full computation is only required in one case. This new approach works inductively on the indices and relies on the observation that once Theorem \ref{th:main} is established for one \emph{fixed} set of indices and \emph{all} suitable subgroups $\Gamma$, the validity of Theorem \ref{th:main} can be extended other sets of indices using the freedom in $\Gamma$. If $g=\exp  \eps X\in G$ for some $X\in\mathfrak g$ is close to the identity we can apply Theorem \ref{th:main} to $\Gamma_g=g\Gamma g^{-1}$ and linearize equation \eqref{eq:main} in $\eps$. Equation \eqref{eq:main} for other choices of indices can then be recovered as equality between lower order terms for suitably chosen $X$.

A simplified weaker version of Theorem \ref{th:main} will be useful to us. Under the assumptions of the Theorem, equation \eqref{eq:c=0} can be replaced by \begin{multline}\label{eq:simplemain}\sum_{\gamma\in\Gamma, |\gamma|<T}Y_{a'b'}(k_1(\gamma))\overline{Y_{ab}(k_2^{-1}(\gamma))}=\frac{\pi  T^{2\delta}}{\delta(\delta-1)} \overline{\int_{K/M}Y_{ab}(k)d\nu_\Gamma(k)} \int_{K/M}Y_{a'b'}(k)d\nu_\Gamma(k) +\\+O\left(T^{2\frac{10\delta+s_1}{11}+\eps}(a+1)^{15/11}(a'+1)^{15/11}\right).\end{multline} In this version we incorporate the lower order main terms into the error term.

%Theorem \ref{th:main} can be applied to obtain a count with power savings for any --- not necessarily bi-$K$-invariant --- norm on $G$. We show how this is done for the Apollonian circle packing problem.  

%\begin{wrapfigure}[15]{r}{0.39\textwidth}\includegraphics{Apollonian_packing_ideal}\caption{\label{fig:apollonian_packing_ideal}Apollonian packing in an ideal triangle.}\end{wrapfigure}
\section{Application to Apollonian circle packing problem.} 
Figures \ref{fig:boundedpacking} and \ref{fig:unboundedpacking} show integral Apollonian circle packings. A thorough introduction to this subject can be found in \cite{lagarias1}. For a bounded packing $P$ let $N^P(T)$ be the number of circles in the packing having curvature at most $T$. For a periodic packing $N^P(T)$ is the number of such circles in one period.

The growth of this function of $T$ was analyzed by Kontorovich and Oh \cite{kontorovich_apollonian_2011} who proved the asymptotic formula $$N^P(T)\sim c_P T^\delta.$$ Here $c_P$ is some constant depending on $P$ and $\delta$ is the Hausdorff dimension of the packing. Since any packing can be sent to any other by a M\"obius transformation, the dimension is a universal constant $\delta$, which McMullen computed to be $\approx1.30568$  \cite{mcmullen_hausdorf_1998}. The approach of \cite {kontorovich_apollonian_2011} is to relate the above problem to counting lattice points in the unit tangent bundle of an infinite volume hyperbolic 3-manifold $\Gamma\quot T^1\H^3$, where $\Gamma$ is a finite index subgroup of the Apollonian group (see \cite{kontorovich_apollonian_2011, lagarias1}). 

Let $Q_D$ be the Descartes form $$Q_D(a,b,c,d)=a^2+b^2+c^2+d^2-\frac12 (a+b+c+d)^2;$$ it has signature $(3,1)$. To each packing one can associate $v\in\R^4$ with $Q_D(v)=0$. The four entries of $v$ are the curvatures of the four circles (with usual conventions regarding signs and orientation)  that generate the packing as an orbit of $v$ under the action of $\Gamma$. Let  $v=gu=g(0,0,1,1)^T$ for some $g\in\O_{Q_D}(\R)$ and let $ \Beta=g^{-1}\Gamma g$; also let $\Beta'=(S_1 g)^{-1}\Gamma(S_1g)$ where $S_1$ is a generator of the Apollonian group \eqref{eq:generators}. 

In this setting we use Theorem \ref{th:main} to prove  power savings in the Apollonian circle packing problem and compute the overall constant\footnote{Added in print: Lee and Oh \cite{lee_oh_effective} have independently obtained a version of this statement using different methods.}. 

\begin{theorem}\label{th:apollonian}For every  $\eps>0$ we have
\begin{equation}N^P(T)=c_PT^\delta+O_\eps(T^{\frac{128\delta+s_1}{129}+\eps}).\end{equation}
\end{theorem}

\begin{remark}
The constant $c_P$ can be expressed in terms of Patterson-Sullivan measure  as 
\begin{multline}c_P=\frac{\pi}{\delta(\delta-1)}\int_{(\Beta\cap N)\quot\d \H^3}(|z|^2+1)^{\delta} d\nu_{\Beta}(z)\int_{\d \H^3} \frac{d\nu_{\Beta}(k)}{\|gku\|^\delta_{L^\infty}}+\\+\frac{\pi}{\delta(\delta-1)}\int_{(\Beta'\cap N)\quot\d \H^3}(|z|^2+1)^{\delta} d\nu_{\Beta'}(z)\int_{\d \H^3} \frac{d\nu_{\Beta'}(k)}{\|S_1gku\|^\delta_{L^\infty}}.\end{multline}
\end{remark}

\begin{remark}
The constant in front of the main term of this expression is several notation changes away from the similar expression in \cite{kontorovich_apollonian_2011}. We use column vectors $v$ instead of row vectors, related by transposition. The quantity $g$ from the present paper corresponds to $(g_0^{-1})^T$ in \cite{kontorovich_apollonian_2011}, and $u$ is $v_0^T$. 
\end{remark}

%\begin{remark}
%The exponent $\eps$ depends on the spectral gap for $\Gamma$. It was conjectured by Sarnak \cite{sarnak_letter_to_lagarias}  that for the Apollonian group $d=0$. That is, the Laplace operator on $\Gamma\quot G$ has only one eigenvalue below 1, and it is equal to $\delta(2-\delta).$   If this conjecture holds, the the spectral gap is $1-\delta(2-\delta)$, and we can take $\eps=\frac{\delta-1}{132}.$
%\end{remark}

More generally let $\vartheta\colon\PSL(2,\C)\to\SO_Q^\circ(\R)$ be a fixed isomorphism (used implicitly) for a quadratic form $Q$ of signature $(3,1)$. Let $\Gamma<\SO_Q^\circ(\R)$ be a discrete subgroup subject to the hypotheses of Theorem \ref{th:main}. Let $u\in\{Q=0\}$ be such that $\Stab u=NM$, let $v=gu,$ and let $ \Beta=g^{-1}\Gamma g$. Also let $\|\cdot\|$ be a norm on $\R^4$.

\begin{theorem}\label{th:general_count}
For every $\eps>0$ we  have
\begin{multline}\sum_{\substack{{\|\gamma v\|<T}\\{\gamma\in\Gamma}}}\chi_{G/(\Stab v\cap \Gamma)}(\gamma)=\left[\frac{\pi}{\delta(\delta-1)}\int_{(\Beta\cap N)\quot\d \H^3}(|z|^2+1)^{\delta} d\nu_{\Beta}(z)\int_{\d \H^3} \frac{d\nu_{\Beta}(k)}{\|gku\|^\delta}\right]T^\delta+\\+O_\eps(T^{\frac{128\delta+s_1}{129}+\eps}).\end{multline}
\end{theorem}

\begin{remark}
In fact $\|\cdot\|$ only needs to satisfy the scaling property of the norm together with non-vanishing and continuity; the full power of the triangle inequality is not needed.
\end{remark}

\begin{remark}
We can assume that $v\in\{Q=c\}$ for some $c$ other than zero; a similar result holds in this case, but then $\Stab v\cong K$ or $\Stab v\cong \SO^\circ(2,1)$. 
\end{remark}

Theorem \ref{th:main} allows for a further refinement of the Apollonian circle packing problem. Consider the packing in an ideal triangle as shown in Figure \ref{fig:apollonian_packing_ideal}. It is contained in the region bounded by three tangent circles as shown; it is also contained in the circle through the vertices of the triangle. Let $S_4$ denote reflection in this circle and  let $v\in \R^4$ consist of the curvatures of the four largest circles of the packing ordered so that $S_4$ preserves the three largest circles. Let $G_4$ be the disk bounded by the circumscribed circle of the triangle. 

\begin{theorem}\label{th:ideal}
With notation as in Theorem \ref{th:apollonian} we have $$N^P(T)=c_P T^\delta+O_\eps(T^{\frac{128\delta+s_1}{129}+\eps})$$ where \begin{multline}c_P=\frac{\pi}{\delta(\delta-1)}\int_{G_4}(|z|^2+1)^{\delta} d\nu_{\Beta}(z)\int_{\d \H^3} \frac{d\nu_{\Beta}(k)}{\|gku\|^\delta_{L^\infty}}+\\+\frac{\pi}{\delta(\delta-1)}\int_{G_4}(|z|^2+1)^{\delta} d\nu_{\Beta'}(z)\int_{\d \H^3} \frac{d\nu_{\Beta'}(k)}{\|S_1gku\|^\delta_{L^\infty}}.\end{multline}
\end{theorem}

%\begin{remark}
%We have again $$c_n=\frac{\pi}{\delta(\delta-1)}\int_{(\Beta\cap N)\quot\d \H^3}(|z|^2+1)^{\delta} d D_{\Beta,n}(z)\int_{\d \H^3} \frac{d D_{\Beta,n}(k)}{\|gku\|^\delta}.$$
%\end{remark}

%Here $\nu_{g_0}=\nu_{\Gamma_{g_0}}=(g_0)_\ast\nu_\Gamma$ is the Patterson-Sullivan measure corresponding to $g_0^{-1}\Gamma g_0.$ 

%% file: parametrization.tex
\chapter{Parametrization}

In this section we recall classical facts about the hyperbolic space and related Lie groups. 

\section{The group $G$.}
The group $\PSL(2,\C)$ is realized as the group of $2\times 2$ matrices with complex values having determinant 1, modulo the center $\{\pm 1\}$. This group is isomorphic to $\SO^\circ(x^2+y^2+z^2-w^2)\cong \SO^\circ(3,1)$. 
An explicit map from $\PSL(2,\C)$ to $\SO^\circ(x^2+y^2+z^2-w^2)$ is given by \small\begin{multline}\label{eq:iota}\iota\colon \pm2\begin{pmatrix}a&b\\c&d\end{pmatrix}\mapsto\\\!\!\!\left(\!\!\!\begin{array}{c@{\,}c@{\!}c@{\,}c}
|a|^2+|d|^2-|b|^2 -|c|^2  &   b \bar{a} +a \bar{b} -d \bar{c} -c \bar{d}  &  i (b \bar{a} -a \bar{b} -d \bar{c} +c \bar{d} ) &|a|^2+|b|^2-|c|^2 -|d|^2\\
 c \bar{a} -d \bar{b} +a \bar{c} -b \bar{d}  &   d \bar{a} +c \bar{b} +b \bar{c} +a \bar{d}  & i (d \bar{a} -c \bar{b} +b \bar{c} -a \bar{d} ) &c \bar{a} +d \bar{b} +a \bar{c} +b \bar{d} \\
i (-c \bar{a} +d \bar{b} +a \bar{c} -b \bar{d} ) &   i (-d \bar{a} -c \bar{b} +b \bar{c} +a \bar{d} ) & d \bar{a} -c \bar{b} -b \bar{c} +a \bar{d} & i (-c \bar{a} -d \bar{b} +a \bar{c} +b \bar{d} ) \\
 |a|^2+|c|^2-|b|^2 -|d|^2 &   b \bar{a} +a \bar{b} +d \bar{c} +c \bar{d}  &  i (b \bar{a} -a \bar{b} +d \bar{c} -c \bar{d} ) & |a|^2+|b|^2+|c|^2+|d|^2 \\
\end{array}\!\!\!\right)\end{multline} \normalsize (cf. \cite{lagarias1}). We suppress $\iota$ if no confusion can arise. 

In this realization of $\SO(3,1)=\SO(x^2+y^2+z^2-w^2)$ take the norm $|\cdot|$ to be \beq\label{eq:norm}|g|=|a(g)|=\left|\begin{pmatrix}\ch t &&&\sh t\\&1\\&&1\\\sh t&&&\ch t\end{pmatrix}\right|= e^t.\eeq It corresponds to $$\left|\begin{pmatrix} e^{t/2}\\&e^{-t/2}\end{pmatrix}\right|=\max(e^{t/2},e^{-t/2})$$ on $\PSL(2,\C).$ Therefore the statement of Theorem \ref{th:main} for $\SO(3,1)$ would have $T^2$ replaced by $T$ on the right hand side. 

We set $$\Theta=\begin{pmatrix}&1\\-1\end{pmatrix},\quad \Phi=\begin{pmatrix}i\\&-i\end{pmatrix}\in\Su(2,\C),\quad h=\begin{pmatrix}1\\&-1\end{pmatrix}\in\Sl(2,\C)$$
\begin{align*}A^+&=\left\{\begin{pmatrix}e^{t/2}\\&e^{-t/2}\end{pmatrix}\colon t\ge 0\right\}=\{\exp(t h/2)\colon t\ge 0\},\quad \\
N&=\left\{\begin{pmatrix}1&z\\&1\end{pmatrix}\colon z\in\C\right\}\\
M&=\left\{\begin{pmatrix}e^{i\phi/2}\\&e^{-i\phi/2}\end{pmatrix}\colon \phi\in [0,2\pi)\right\}\\
K&=\left\{\exp(\phi \Phi/2)\exp(\theta \Theta/2) \exp(\phi_2 \Phi/2)\colon \phi\in[0,\pi),\phi_2\in[0,2\pi),\theta\in[0,\pi)\right\}.\end{align*} The latter parametrization of $K$ is called the Euler angle parametrization. The Haar measure on $K=\PSU(2)$ in these coordinates is given by $$dk = \frac1{4\pi^2}\sin \theta\, d\phi\, d\theta\, d\phi_2,$$ and the Haar measure on $A^+$ is \beq da_t= \sh^2t\, dt.\label{eq:haar_a}\eeq On the manifold $K/M$ we get the measure $$dk=\frac1{4\pi}\sin\theta \,d\phi \,d\theta;$$ we will use $dk$ to denote the measure both on $K/M$ and $K$. Thus we have almost everywhere unique decomposition 
\beq g=\exp(\phi \Phi/2)\cdot\exp(\theta \Theta/2) \cdot\exp(\psi \Phi/2)\cdot\exp(t h/2)\cdot\exp(\theta_2 \Theta/2) \cdot\exp(\phi_2 \Phi/2).\label{eq:coordinates}\eeq The Haar measure on $\PSL(2,\C)$ in these coordinates is \beq 4\pi dk_1 da_t dk_2\label{eq:haar}\eeq  Here $dk_1$ is the probability measure on $K/M$, and $dk_2$ is the probability measure on $K$. The normalization of \eqref{eq:haar} is actually quite natural in view of the relation of $G$ to hyperbolic space, as we explain below. 

\section{Hyperbolic space.} We use two models for $\H^3$, the upper half plane and the ball model (see \cite{EGM} for details). The upper half space model is realized as $\{x_1+ix_2+jy\mid x_1,x_2\in\R,y>0\}.$ The hyperbolic metric and volume are $$ds^2=\frac{dx_1^2+dx_2^2+dy^2}{y^2}\quad dV=\frac{dx_1\,dx_2\,dy}{y^3}.$$ The ball model $\B^3$ is realized as the unit ball $ \{x_1+ix_2+jy\colon x^2_1+x_2^2+y^2< 1\}\subset \R^3$. We use the isomorphism $$z\mapsto (z-j)(-jz+1)^{-1}$$ from $\H^3$ to $\B^3$ to transfer the metric. In these coordinates the metric and volume elements are $$ds^2=4\frac{dx_1^2+dx_2^2+dy^2}{(1-x_1^2-x_2^2-y^2)^2}\quad dV=8 \frac{dx_1\,dx_2\,dy}{(1-x_1^2-x_2^2-y^2)^2}.$$ We also identify boundaries $\d\H^3$ and $S^2\cong K/M$ by extending the map there. 

In the ball model there is a convenient set of coordinates that match the action of $G$. Any point in $\B^3=G/K$ is the image of $0$ under an element of $A^+$ and an element of $K/M$. To wit, \begin{multline}\begin{pmatrix}e^{i\phi/2}\\&e^{-i\phi/2}\end{pmatrix}\begin{pmatrix}\cos\frac\theta2&\sin\frac\theta2\\-\sin\frac\theta2&\cos\frac\theta2\end{pmatrix}\begin{pmatrix}e^{t/2}\\&e^{-t/2}\end{pmatrix}0=\\=\th\frac t2(-\cos\phi\sin\theta-i\sin\phi\sin\theta+j\cos\theta).\end{multline} In these coordinates the volume element reads $$dV=\sin\theta d\phi\,d\theta\cdot \sh^2 t\,dt=4\pi dk_1 da_t.$$ This measure extends to all of $G$ as in \eqref{eq:haar}. The advantage of these normalizations is that respect the structure of $G$ as a $K$-bundle over $\H$. In particular, the natural map $L^2(\H)\to L^2(G)$ is an isometry.

%% file: line_model.tex
\chapter{Line model}

The goal of this section is to understand the action of $\mathfrak g^\C$ on the $K$-types of a complementary series representation. Neither the formula from Lemma \ref{lem:generalform}   nor the action of certain operators on the $K$-types \eqref{eq:haction}, \eqref{eq:Jmaction}, \eqref{eq:Jpaction} is widely available in the literature. Operators similar to those we introduce appeared in \cite{kelmer_logarithm_2011} without normalizations. 

\section{Complementary series for $G$.} Here we review the complementary series representations of $G$ in the line model \cite{knapp_representation_theory, gelfand_teoriya_1966}. Let $0<s<2$ and let $G$ act on functions $f\colon\C\to \C$ by $$\pi_s\begin{pmatrix}a&b\\c&d\\\end{pmatrix}.f(z)=\frac1{|bz+d|^{2s}}f\left(\frac{az+c}{bz+d}\right).$$ Denote by $I$ the intertwining operator \beq\label{eq:intertwining}I.f(\zeta)=\int_\C \frac{f(z)}{|z-\zeta|^{2(2-s)}}\frac i2 dz\wedge d\bar z.\eeq The action above gives a unitary representation of $G$ on the space of functions from $\C$ to $\C$ equipped with the inner product $$\langle f_1,f_2\rangle=\int_\C f_1(z) \overline{I.f_2(z)}\frac i2 dz\wedge d\bar z;$$ call this space $V_s$.  Then the complementary series representation of $G$ with parameter $s$ is $(\pi_s,V_s)$; it is non-tempered and irreducible. 

\section{Lie algebra of $G$.} We parametrize the (real six-dimensional) Lie algebra $\mathfrak g=\Sl(2,\C)\cong \So(3,1)$ by $$
h=\begin{pmatrix}1\\&-1\end{pmatrix}\quad
e=\begin{pmatrix}\phantom{0}&1\\&\end{pmatrix}\quad
f=\begin{pmatrix}&\phantom{0}\\1\end{pmatrix}$$
and
$$
\Ih=\begin{pmatrix}i\\&-i\end{pmatrix}\quad
\Ie=\begin{pmatrix}\phantom{0}&i\\&\end{pmatrix}\quad
\If=\begin{pmatrix}&\phantom{0}\\i\end{pmatrix}.$$  Note that the letter ``$i$'' is concatenated to the symbol it follows.

%The commutator table for $\mathfrak g$ is $$
%\begin{array}{c||c|c|c||c|c|c|} 
%&h&e&f&ih&ie&if\\\hline &&&&&&\\[-2.5ex]\hline
%h&0&2e&-2f&0&2ie&-2if\\\hline
%e&-2e&0&h&-2ie&0&ih\\\hline
%f&2f&-h&0&2if&-ih&0\\\hline &&&&&&\\[-2.5ex]\hline
%ih&0&2ie&-2if&0&-2e&2f\\\hline
%ie&-2ie&0&ih&2e&0&-h\\\hline
%if&2if&-ih&0&-2f&h&0\\\hline
%\end{array}
%$$

\section{Universal Enveloping Algebra of $G$.} Let $\mathfrak g^\C$ denote the complexification of $\mathfrak g$; it is a complex six-dimensional Lie algebra whose field of scalars $\C$ we write as $a+b\I$, $a,b\in\R$. As usual, we denote its universal enveloping algebra by $U(\mathfrak g^\C)$ and lift the representation $\pi_s$ to it. We shall indulge in the sin of referring to $\pi_s$ in this way whether it is a representation of $G$, $\mathfrak g$, or $U(\mathfrak g^\C);$ we shall also drop the subscript $s$ in most cases. 

The center $Z$ of $U(\mathfrak g^\C)$ is one-dimensional and is generated by the element $$\Omega= h^2-\Ih^2+2(ef+fe-\Ie \If-\If \Ie);$$ it is called the Casimir operator and has the property that $$\pi(\Omega).F+4s(2-s)F=0$$ for smooth $F\in V_s$. 

%1/2 hh - 1/2 ihih + ef + fe - ieif - ifie

\section{$K$-type decomposition.} Consider the decomposition of $(\pi_s,V_s)$ into $K$ types. That is, restrict $\pi$ to $K<G$ and decompose $V_s$ into irreducible representations of $K$. The group $K$ has one irreducible representation in every odd dimension. The  multiplicity of each irreducible representation in $\pi|_K$ is 1. Let $\{v_{lj}^s\colon l\ge 0, |j|\le l\}$ be an orthonormal basis for $V_s$ subject to the conditions \beq\label{eq:system1}\pi(\Omega_K)v_{lj}^s+l(l+1)v_{lj}^s=0\eeq \beq\label{eq:system2}\pi(\Ih)v_{lj}^s=2ijv_{lj}^s,\eeq where $\Omega_K=\frac14(\Ih^2+(\Ie+\If)^2+(e-f)^2)$ is the Casimir element for $K$. Thus for each $l$ the set $\{v_{lj}^s\colon |j|\le l\}$ is a basis for the ($2l+1$)-dimensional irreducible representation of $K$. This allows one to recover exact formulae for $v_{lj}^s$ as solutions to the system of ordinary differential equations \eqref{eq:system1}--\eqref{eq:system2}. We will suppress the index $s$ if no confusion can arise.

\section{Useful operators in $\mathfrak g$.} The following elements of $U(\mathfrak g^\C)$ will be important for our investigation: 
\begin{align}J^{\pm}&=f\pm \I if\\ R&=\Ie+\If+\I(e-f)\\ L&=\Ie+\If-\I(e-f).\end{align}
The operators $R$ and $L$ (raising and lowering, resp.) are the usual ladder operators for $\Su(2)$ and map $v_{lj}$ to multiples of $v_{l,j+1}$ and $v_{l,j-1}$, respectively.

\begin{lemma}
The action of $J^\pm$ satisfies  \beq\label{eq:constantscl}\pi(J^+) v_{ll}=C(l) v_{l+1,l+1}\eeq and $$\pi(J^-)v_{l,-l}=C(l) v_{l+1,-l-1}$$ for certain normalizing constants $C(l)$.
\end{lemma}

\begin{proof}

%One addition to the previous write-up: $f+\I if$ \emph{is} the jumping operator: $$(f+\I if)^k \phi_0=\const\cdot f_{kk}$$ and similarly for the operator $f-\I if$. 
We only prove the first statement; the proof of the second one is similar. 

The computation lies entirely in the Lie algebra. Define the following elements of $\Su(2)$: $$X=\begin{pmatrix}i\\&-i\end{pmatrix}\quad Y=\begin{pmatrix}&1\\-1\end{pmatrix}\quad Z=\begin{pmatrix}&i\\i\end{pmatrix}$$ and the following elements from $\Sl(2,\C)$: $$iX=\begin{pmatrix}-1\\&1\end{pmatrix}\quad iY=\begin{pmatrix}&i\\-i\end{pmatrix}\quad iZ=\begin{pmatrix}&-1\\-1\end{pmatrix}.$$ The Casimir operator for $\Su(2)$ can be written as $$\Omega_K=\frac14(X^2+Y^2+Z^2).$$ With this normalization we have $$\pi(\Omega_K )v_{lj}=-l(l+1)v_{lj}$$ for every $|j|\le l$. Thus to claim that $f+\I \If$ is the correct operator we need two conditions to hold: \beq\label{eq:casimir_condition}\pi([\Omega_K,f+\I \If])v_{ll}=-2(l+1)v_{ll}\eeq and \beq\label{eq:X_condition}\pi([X/2,f+\I \If])v_{ll}=\I(f+\I \If)v_{ll}.\eeq This first condition ensures that the Casimir eigenvalue changes from $-l(l+1)$ to $-(l+1)(l+2)$, while the second condition guarantees that the eigenvalue of $X/2$ is increased by $\I$. Since we assume we start with $v_{ll}$, our initial assumptions are \begin{align} \pi(\Omega_K) v_{ll}&=-l(l+1)v_{ll},\\\pi(X)v_{ll}&= 2\I lv_{ll},\\ \pi(R)v_{ll}&=\pi(Z+\I Y)v_{ll} =0.\end{align} Thus we compute \begin{multline}[\Omega_K,f+\I \If]=\\=\left[\frac{X^2+Y^2+Z^2}{4},\frac{-Y-iZ+\I Z-\I iY}{2}\right]=\frac12(X(iY-\I iZ)-iX(Y-\I Z))=\\=\frac12(iX \I(Z+\I Y)+(iY-\I iZ)X+[X,iY-\I iZ]).\label{eq:commutator}\end{multline} The first term contains the raising operator $Z+\I Y$, so we drop it. The second term contains $X$, which we replace by $2\I l$. The last term we massage further to get $$ (l+1)(iZ+\I iY).$$ Recalling that $\pi(Z+\I Y)=0$ on $v_{ll}$, we can write \eqref{eq:commutator} as $$-2(l+1)\frac{-Y-iZ+\I Z-\I iY}{2},$$ confirming \eqref{eq:casimir_condition}. The second condition \eqref{eq:X_condition} is verified directly by observing that $$\left[\frac X2,\frac{-Y-iZ+\I Z-\I iY}{2}\right]=\I  \frac{-Y-iZ+\I Z-\I iY}{2}.$$

\end{proof}

These operators provide an easy way of going from one $K$ representation to another, and we call them ``jumping operators'' for this reason.  We have that $v_{lj}$ is a multiple of $$\pi\left( L^{l-j}(J^+)^l\right)v_{00}.$$ 

In polar coordinates $z=re^{i\alpha}$ we can write the operators as \beq\label{eq:Rcoordinates}\pi(R)=e^{i\alpha}\left(2isr+i(r^2+1)\d_r+\left(r-\frac1r\right)\d_\alpha\right)\eeq and 
\beq\label{eq:Lcoordinates}\pi(L)=e^{-i\alpha}\left(-2isr-i(r^2+1)\d_r+\left(r-\frac1r\right)\d_\alpha\right).\eeq For the jumping operators we have \beq\label{eq:Jcoordinates1}\pi(J^+)=-e^{i\alpha}\left(\d_r+\frac{i}{r}\d_\alpha\right)\eeq
and \beq\label{eq:Jcoordinates2}\pi(J^-)=-e^{-i\alpha}\left(\d_r-\frac{i}{r}\d_\alpha\right).\eeq Finally for the last two operators of the basis $\{J^+,J^-,R,L,h,\Ih\}$ we have \beq\label{eq:hcoordinates}\pi(h)=2s+2r\d_r\eeq and \beq\label{eq:ihcoordinates}\pi(\Ih)=2 \d_\alpha.\eeq

\section{Normalizing operators.} In this section we establish how the operators $J^\pm, R,L$ defined above affect the $L^2$ norm of $v_{lj}$ and compute $v_{lj}$. It is well-known that \beq \pi(R)v_{lj}=2\sqrt{(l-j)(l+j+1)}v_{l,j+1}\label{eq:Rnormalization}\eeq and \beq\label{eq:Lnormalization}\pi(L)v_{lj}=2\sqrt{(l+j)(l-j+1)}v_{l,j-1}.\eeq In order to calculate the constants $C(l)$ from \eqref{eq:constantscl} we need to compute $v_{ll}$ and the constant for the intertwining operator.

\begin{lemma}
We have that $v_{ll}$ is a multiple of $$r^le^{il\alpha}(1+r^2)^{-s-l}$$
\end{lemma}

\begin{proof}
Follows from solving the system of differential equations \eqref{eq:system1}--\eqref{eq:system2}.
\end{proof}

\begin{lemma}
The action of the intertwining operator $I\colon V_s\to V_{2-s}$ satisfies $$I.v_{ll}^s=(-1)^l\pi \frac{(\Gamma(s-1))^2}{\Gamma(l+s)\Gamma(s-l-1)}v_{ll}^{2-s}.$$ 
%Here $v_{ll}$ on the left hand side is understood to be in $V_s$, and $v_{ll}$ on the right hand side is in $V_{2-s}.$
\end{lemma}

\begin{proof}
It is clear that  $$I.v_{ll}^s(r,\alpha)=\const \cdot v_{ll}^{2-s}(r,\alpha)$$ since $I$ intertwines the action of $\pi_s$. The only question is establishing the value of the constant. We plug in $\alpha=0$ and apply $\left.\frac{d^l}{dr^l}\right|_{r=0}$ to both sides, recalling the definition of the intertwining operator \eqref{eq:intertwining}. The result follows. 
\end{proof}

\begin{lemma}\label{lem:line_normalization}
$$v_{l,\pm l}(z)= \frac{\sqrt{(-1)^l\Gamma(l+s)\Gamma(s-l-1) (2l+1)!}}{l! \pi\Gamma(s-1)} r^l e^{\pm il\alpha}(1+r^2)^{-s-l}.$$ 
\end{lemma}

\begin{proof}
Let $$v_{ll}(r,\alpha)=b_l r^le^{il\alpha}(1+r^2)^{-s-l}$$ for $b_l>0$. Then $$1=\langle v_{ll},v_{ll}\rangle=\int_\C v_{ll}(z) \overline{I.v_{ll}(z)}\frac i2 dz\wedge d\bar z$$ and from the previous Lemma we have $$1=|b_l|^2(-1)^l\pi \frac{(\Gamma(s-1))^2}{\Gamma(l+s)\Gamma(s-l-1)}\int_{r=0}^\infty (1+r^2)^{-s-l}(1+r^2)^{s-2-l}r dr.$$ Evaluating the integral and solving for $b_l$ yields the desired result. 
\end{proof}

%The precise form of $v_{ll}$ allows us to normalize the jumping operators. Using the formulae for these operators in coordinates \eqref{eq:Jcoordinates1}--\eqref{eq:Jcoordinates2} we get \beq\label{eq:Jnorm}(J^+)^lv_{00}=2^ll!\sqrt{\binom{s+l-1}{2l+1}\frac{(-1)^l}{s-1}}v_{ll}.\eeq 
The resulting expressions for general $v_{lj}$ are given the next 

\begin{lemma}\label{lem:generalform}
For $j\ge 0$ we have \begin{multline} \label{eq:vlj}v_{lj}(r,\alpha)=\frac{\sqrt{(-1)^l\Gamma(l+s)\Gamma(s-l-1)(2l+1)(l-j)!(l+j)!}}{l!\pi \Gamma(s-1)} \\ \frac{e^{ij\alpha}}{(r^2+1)^{s+l}}\sum_{k=0}^{l-j}r^{2(l-k)-j}\binom l{l-j-k}\binom lk (-1)^k.\end{multline} A similar formula holds for $j<0.$
\end{lemma}

\begin{proof}
This formula is verified by induction on $j$. The base case is $j=l$, which is Lemma \ref{lem:line_normalization}. The induction step is going from $j$ to $j-1$; it is carried out by applying the operator $L$ from \eqref{eq:Lcoordinates} with normalization from \eqref{eq:Lnormalization}.
\end{proof}

Lemma \ref{lem:generalform} above allows us to characterize the action of $\mathfrak g^\C$ on the $K$-types. By direct substitution we can verify that 

\small

\begin{align}\label{eq:haction} \frac12 \pi(h)v_{lj}&=\sqrt{\frac{(l+1-s)(l-1+s)(l^2-j^2)}{(2l-1)(2l+1)}} v_{l-1,j}-\sqrt{\frac{(l+2-s)(l+s)((l+1)^2-j^2)}{(2l+1)(2l+3)}}v_{l+1,j}\\
 \label{eq:Jpaction} \pi(J^+) v_{lj}&=\sqrt{\frac{(l+2+j)(l+1+j)(s+l)(l+2-s)}{(2l+1)(2l+3)}}v_{l+1,j+1}-\sqrt{(l+j+1)(l-j)} v_{l,j+1}+\\
&+\sqrt{\frac{(l-j)(l-j-1)(l-1+s)(l+1-s)}{(2l-1)(2l+1)}} v_{l-1,j+1}\notag\\
 \label{eq:Jmaction} \pi(J^-) v_{lj}&=\sqrt{\frac{(l+2-j)(l+1-j)(s+l)(l+2-s)}{(2l+1)(2l+3)}}v_{l+1,j-1}-\sqrt{(l-j+1)(l+j)} v_{l,j-1}+\\
&+\sqrt{\frac{(l+j)(l+j-1)(l-1+s)(l+1-s)}{(2l-1)(2l+1)}} v_{l+1,j-1}\notag
\end{align}

\normalsize

%In addition it will be useful to know that $$v_{10}(z)=\frac1\pi\sqrt{\frac{3s(s-1)}{2-s}}(r^2-1)(r^2+1)^{-s-1}.$$

%\subsection{Examples.} We  compute $v_{lj}$ for several values of indices. Writing $z=re^{i\alpha}$ in polar coordinates we have $$v_{00}(z)=\frac{\sqrt{s-1}}\pi (r^2+1)^{-s}$$ and more generally $$v_{l,\pm l}(z)= \frac{\sqrt{(-1)^l\Gamma(l+s)\Gamma(s-l-1) (2l+1)!}}{l! \pi\Gamma(s-1)} r^l e^{\pm il\alpha}(1+r^2)^{-s-l}.$$ In addition it will be useful to know that $$v_{10}(z)=\frac1\pi\sqrt{\frac{3s(s-1)}{2-s}}(r^2-1)(r^2+1)^{-s-1}.$$

%% file: automorphic_model.tex
\chapter{Automorphic model}

Now we need to review some representation theory of $G$ in the automorphic model. All results are standard, except for formulae from Section \ref{sec:kak} that are not to the best of our knowledge. 

\section{Spectral decomposition.} 
Consider the decomposition of the right regular representation of $G$ on $$L^2(\Gamma\quot G)=V_{\phi_0}\oplus V_{\phi_1}\oplus\dots\oplus V_{\phi_d}\oplus V_{\text{temp}}.$$ Each $V_{\phi_n}$ is isomorphic to a complementary series representation of $G$ with parameter $s_n$ and is irreducible; $V_{\text{temp}}$ is the reducible part consisting of the tempered representations (see \cite{gelfand_teoriya_1966,knapp_representation_theory}). Each $V_{\phi_n}$ contains exactly one positive $K$-fixed vector, up to scalar multiplication; call this vector $\phi_n$. These functions are well-defined on $\Gamma\quot \H^3$, and satisfy $$\Delta \phi_n+s_n(2-s_n)\phi_n=0,$$ where $\Delta$ is the Laplace-Beltrami operator on the corresponding hyperbolic manifold. 

\section{Patterson-Sullivan theory.\label{sec:ps}} The function $\phi_0$ has the smallest eigenvalue and is called the base eigenfunction; it can be realized explicitly as the integral of the Poisson kernel (raised to the power $\delta$) against the Patterson-Sullivan measure $\nu$. This measure is supported on a Cantor set $\Lambda\subset\d\H^3\cong S^2$ and has Hausdorff dimension $\delta$ (see \cite{patterson_limit_set, sullivan_entropy}). That is, the connection between $\nu$ and $\phi_0$ is \beq\label{eq:ps_poisson}\phi_0(\phi,\theta,r)=\int_{(v,u)\in S^2}(P(\phi,\theta,r;v,u))^\delta d\nu (v,u),\eeq where \beq\label{eq:poisson_kernel}P(\phi,\theta,r;v,u)=\frac{1-r^2}{1-2r(\sin\theta\cos\phi\sin u\cos v +\sin\theta\sin\phi \sin u\sin v+\cos\theta\cos u)+r^2}\eeq is the Poisson kernel in the Euler angle parametrization. The variable $r$ is the distance in the disk model with corresponds to $\th \frac t2$ in our coordinates on $A^+$.

We can also view the hyperbolic three space in the upper half-plane model $x_1+ix_2+jy$ with real $x_1$, $x_2$, $y$ as in \cite{EGM}. The action of $G$ on $\H^3$ is by M\"obius transformations with quaternion multiplication. In these coordinates we have $$\phi_0(x_1+ix_2+jy)=\int_{z\in\C\cup\{\infty\}} \left(\frac{y(1+|z|^2)}{|z-(x_1+ix_2)|^2+y^2}\right)^\delta d\nu(z).$$ 

There is a similar statement relating eigenfunctions $\phi_n$ with eigenvalues $s_n$ for $n\ge 1$ to certain distributions $D_{\Gamma,n}$, not measures, on $\d\H^3$ via the Poisson-Helgason transform \cite{grellier_otal_bounded, nalini_zelditch_patterson} (see \cite{helgason_groups_geometric} for  a discussion of the $\SL(2,\R)$ case). We have \beq\label{eq:distributions}\phi_n(\phi,\theta,r)=D_{\Gamma,n}((P(\phi,\theta,r;\cdot))^{s_n}).\eeq In \cite{grellier_otal_bounded} it is proven that $D_{\Gamma,n}$ is in the dual space of $C^{2-s_n}(S^2).$ 

\section{Lie algebra elements in $KA^+K$ coordinates.\label{sec:kak}} We will subsequently need  the action of $U(\mathfrak g^\C)$ in $KA^+K$ coordinates; specifically, in coordinates from \eqref{eq:coordinates}. This equation is a smooth map from $\R^6$ to $G$. Its derivative is a map from $\mathrm{Lie}(\R^6)$ to $\mathfrak g$. The derivative map is invertible, and we can write the basis $\{h,\Ih,e,\Ie,f,\If\}$ in terms of the basis $\{\d_\phi,\d_\theta,\d_\psi,\d_t,\d_{\theta_2},\d_{\phi_2}\}.$ We suppress the lengthy computation and only cite the final result. 

\small
\begin{multline}
\frac h2=-\cth t \sin \theta_2 \d_{\theta_2}+\cos \theta_2 \d_{t}-\ctg \theta \csch t \sin \psi \sin \theta_2 \d_{\psi}+\\+\cos \psi \csch t \sin \theta_2 \d_{\theta}+\cosec \theta \csch t \sin \psi \sin \theta_2 \d_{\phi}\label{eq:formulae}
\end{multline}
\begin{multline}
\frac{\Ih}2=\d_{\phi_2}\\
\end{multline}
\begin{multline}
\frac e2=-\frac{1}{2} (\ctg \theta_2+\cth t \cosec \theta_2) \sin \phi_2 \d_{\phi_2}+\frac{1}{2} \cos \phi_2 (1+\cos \theta_2 \cth t) \d_{\theta_2}+\frac{1}{2} \cos \phi_2 \sin \theta_2 \d_{t}+\\+\frac{1}{8} e^{-t} \cosec \frac{\theta_2 }{2} \csch t \sec \frac{\theta_2 }{2}((-1+\cos \theta_2) \sin \phi_2+e^{2 t} (1+\cos \theta_2) \sin \phi_2+\\+2 e^t \ctg \theta (\cos \psi \sin \phi_2+\cos \phi_2 \cos \theta_2 \sin \psi) \sin \theta_2) \d_{\psi}+\\+\frac{1}{2} \csch t (-\cos \phi_2 \cos \psi \cos \theta_2+\sin \phi_2 \sin \psi) \d_{\theta}-\\-\frac{1}{2} \cosec \theta \csch t (\cos \psi \sin \phi_2+\cos \phi_2 \cos \theta_2 \sin \psi) \d_{\phi}
\end{multline}
\begin{multline}
\frac{\Ie}2=-\frac{1}{2} \cos \phi_2 (\ctg \theta_2+\cth t \cosec \theta_2) \d_{\phi_2}-\frac{1}{2} (1+\cos \theta_2 \cth t) \sin \phi_2 \d_{\theta_2}-\\-\frac{1}{2} \sin \phi_2 \sin \theta_2 \d_{t}+\frac{1}{2} (\cos \phi_2 (\ctg \theta_2 \cth t+\cosec \theta_2+\cos \psi \ctg \theta \csch t)-\\-\cos \theta_2 \ctg \theta \csch t \sin \phi_2 \sin \psi) \d_{\psi}+\frac{1}{2} \csch t (\cos \psi \cos \theta_2 \sin \phi_2+\cos \phi_2 \sin \psi) \d_{\theta}+\\+\frac{1}{2} \cosec \theta \csch t (-\cos \phi_2 \cos \psi+\cos \theta_2 \sin \phi_2 \sin \psi) \d_{\phi}
\end{multline}
\begin{multline}
\frac f2=\frac{1}{2} (\ctg \theta_2-\cth t \cosec \theta_2) \sin \phi_2 \d_{\phi_2}+\frac{1}{2} \cos \phi_2 (-1+\cos \theta_2 \cth t) \d_{\theta_2}+\frac{1}{2} \cos \phi_2 \sin \theta_2 \d_{t}+\\+\frac{1}{8} e^{-t} \cosec \frac{\theta_2 }{2} \csch t \sec \frac{\theta_2 }{2}(e^{2 t} (-1+\cos \theta_2) \sin \phi_2+(1+\cos \theta_2) \sin \phi_2+\\+2 e^t \ctg \theta (\cos \psi \sin \phi_2+\cos \phi_2 \cos \theta_2 \sin \psi) \sin \theta_2) \d_{\psi}+\\+\frac{1}{2} \csch t (-\cos \phi_2 \cos \psi \cos \theta_2+\sin \phi_2 \sin \psi) \d_{\theta}-\\-\frac{1}{2} \cosec \theta \csch t (\cos \psi \sin \phi_2+\cos \phi_2 \cos \theta_2 \sin \psi) \d_{\phi}
\end{multline}
\begin{multline}
\frac{\If}2=-\frac{1}{2} \cos \phi_2 (\ctg \theta_2-\cth t \cosec \theta_2) \d_{\phi_2}+\frac{1}{2} (-1+\cos \theta_2 \cth t) \sin \phi_2 \d_{\theta_2}+\\+\frac{1}{2} \sin \phi_2 \sin \theta_2 \d_{t}+\frac{1}{4} e^{-t} \csch t (-\cos \phi_2 (\ctg \theta_2+\cosec \theta_2)-\\-2 e^t \ctg \theta (\cos \phi_2 \cos \psi-\cos \theta_2 \sin \phi_2 \sin \psi)+e^{2 t} \cos \phi_2 \tg\frac{\theta_2 }{2}) \d_{\psi}-\\+\frac{1}{2} \csch t (\cos \psi \cos \theta_2 \sin \phi_2+\cos \phi_2 \sin \psi) \d_{\theta}+\\+\frac{1}{2} \cosec \theta \csch t (\cos \phi_2 \cos \psi-\cos \theta_2 \sin \phi_2 \sin \psi) \d_{\phi}\label{eq:formulae_last}
\end{multline}

\normalsize
These formulae will be useful later.

%% file: decay.tex
\chapter{Decay of matrix coefficients} 

In this section we state well-known results on decay of matrix coefficients. We refer the reader to \cite{maucourant_2007}, \cite{venkatesh_sparse_2010}, \cite{kontorovich_apollonian_2011}, \cite{howe_moore} for proofs. 

Let $\omega=1-\Omega_K\in U(\mathfrak g^\C).$

\addtocounter{section}{1}

\begin{lemma}\label{lem:tempered}Let $(\pi,V)$ be a tempered unitary representation of $G$. Then for any $K$-finite $w_1,w_2\in V$ $$\langle \pi(k_1 a_t k_2).w_1,w_2\rangle \ll t e^{-t} \sqrt{\dim \pi(K).w_1} \|w_1\|_{L^2} \cdot \sqrt{\dim \pi(K).w_2} \|w_2\|_{L^2} .$$ Furthermore, for any $w_1,w_2\in V$ we have $$\langle \pi(k_1 a_t k_2).w_1,w_2\rangle \ll t e^{-t}  \|\omega(w_1)\|_{L^2} \|\omega(w_2)\|_{L^2} .$$
\end{lemma}

\begin{lemma}\label{lem:decay}Fix $1<s_0<2$. Let $(\pi,V)$ be a unitary representation of $G$ that does not weakly contain any complementary series representation $V_s$ for $s>s_0$. Then for any $K$-finite $w_1,w_2\in V$ $$\langle \pi(k_1 a_t k_2).w_1,w_2\rangle \ll  e^{-s_0t} \sqrt{\dim \pi(K).w_1} \|w_1\|_{L^2} \cdot \sqrt{\dim \pi(K).w_2} \|w_2\|_{L^2} .$$ Furthermore, for any $w_1,w_2\in V$ we have $$\langle \pi(k_1 a_t k_2).w_1,w_2\rangle \ll  e^{-s_0t}  \|\omega(w_1)\|_{L^2} \|\omega(w_2)\|_{L^2} .$$
\end{lemma}

%% file: proof_of_main_th.tex
\chapter{Proof of Theorem \ref{th:main}}

%The parametrization of $\SU(2)$ above is by Euler angles. Writing each element as above as $(\phi,\theta,\phi_2)$ we note that $M\quot K=\{(0,\theta,\phi_2)\}$ and $K/M=\{(\phi,\theta,0)\}.$ In addition we remark that under the usual isomorphism of $\PSL(2,\C)$ to $\SO(3,1)$ a ball of radius $T$ is mapped to a ball of radius $T^2$, so that the we will have a main term of $T^{2\delta}$ when counting in $\PSL(2,\C)$ and the same for the error terms. 

\section{Setup.} Let $$N(T)=\sum_{\gamma\in\Gamma, |\gamma|<T}Y_{a'a;b'bc}(K(\gamma)).$$ Also define $f_T(g)$ on $G$ by \beq\label{eq:ftg}f_T(g)=Y_{a'a;b'bc}(K(g))\chi_{|g|<T}.\eeq Let $F_T\colon \Gamma\quot G\times\Gamma\quot G\to\C$ be defined by $$F_T(g,h)=\sum_{\gamma\in\Gamma}f_T(g^{-1}\gamma h).$$ Observe that $N(T)=F_T(e,e).$ 

For a fixed parameter $\eta$ (to be chosen later depending on $T$) let $\psi\colon G\to\R^+$ be a smooth function supported in a ball of radius $\eta$ about $e$ and having integral 1 with respect to the measure we fixed in \eqref{eq:haar}. Let $\Psi\colon \Gamma\quot G\to R^+$ be defined by $$\Psi(x)=\sum_{\gamma\in\Gamma}\psi(\gamma x).$$

Let $$H(T)=\int_{\Gamma\quot G}\int_{\Gamma\quot G}F_T(g,h)\Psi(h)\Psi(g)\,dh\, dg.$$

\begin{lemma}
\beq \label{eq:smoothing}|H(T)-N(T)|\ll (a+1)^{3/2}(a'+1)^{3/2}T^{2\delta}\eta.\eeq
\end{lemma}

\begin{proof}
The left hand side of \eqref{eq:smoothing} is at most  \beq\sum_{\gamma\in\Gamma} \int_{g\in\Gamma\quot G}\int_{h\in\Gamma\quot G}|f_T(g^{-1}\gamma h)-f_T(\gamma)|\Psi(g)\Psi(h)dh\, dg.\label{eq:error}\eeq We distinguish three cases. 
\begin{enumerate}
 \item If $|\gamma|>\frac{T}{1-\eta}$, then both $f_T(g^{-1}\gamma h)$ and $f_T(\gamma)$ vanish.
\item If $\frac{T}{1-\eta}\ge |\gamma|>\frac T{1+\eta}$, then trivially $$|f_T(g^{-1}\gamma h)-f_T(\gamma)|\ll \|Y_{a'a;b'bc}\|_{L^\infty}\ll\sqrt{(a+1)(a'+1)}.$$
\item Finally if $|\gamma|\le \frac T{1+\eta}$, then $$|f_T(g^{-1}\gamma h)-f_T(\gamma)|\ll\eta \|Y_{a'a;b'bc}\|_{C^1}\ll \eta ((a+1)(a'+1))^{3/2}.$$
\end{enumerate}

Thus the error \eqref{eq:error} is at most $$\sqrt{(a+1)(a'+1)}\sum_{\frac{T}{1-\eta}\ge |\gamma|>\frac T{1+\eta}}1+\eta ((a+1)(a'+1))^{3/2}\sum_{|\gamma|\le \frac T{1+\eta}}1\ll\eta ((a+1)(a'+1))^{3/2}T^{2\delta}$$ by Lax-Phillips \cite{lax_phillips} and Theorem \ref{th:lax} above. Specifically, we use the leading term of their main theorem: $$\sum_{\gamma\in\Gamma,|\gamma|<T} 1=\const\cdot T^{2\delta} +O(T^{2\delta-\sigma}),$$ where $\sigma$ depends on the spectral gap for $L^2(\Gamma\quot G)$. 

\end{proof}

The previous Lemma allows us to work with $H$ rather than with $N$, and the next Lemma massages $H$ into a more pleasant form. 
%It is also the worst case scenario that is realized when $b=\pm 1$. For other values of $b$ we can use a higher power of $\eta$ and a higher power of $a$. 

\begin{lemma} We have 
 $$H(T)=\int_{g\in G} f_T(g)\langle \pi(g)\Psi,\Psi\rangle dg.$$
\end{lemma}

\begin{proof}
We compute
$$\int_{g\in\Gamma\quot G}\int_{h\in\Gamma\quot G}F_T(g,h)\Psi(h)\Psi(g)dh\, dg=\sum_{\gamma\in\Gamma} \int_{g\in\Gamma\quot G}\int_{h\in\Gamma\quot G} f_T(g^{-1}\gamma h)\Psi(h)\Psi(g)dh\,dg.$$ In the inner integral substitute $x=g^{-1}\gamma h$, whence $h=\gamma g x$. Using $\Gamma$-invariance of $\Psi$ we get $$\sum_{\gamma\in\Gamma}\int_{g\in\Gamma\quot G}\int_{x\in g^{-1}\gamma(\Gamma\quot G)}f_T(x)\Psi(g x)\Psi(g)dx \,dg.$$ Now the result follows by combining the sum over $\gamma$ with the integral in $x$ and renaming the variables.

\end{proof}

\section{Spectral decomposition.} Write $\Psi=\Psi_0+\dots+\Psi_d+\Psi_{\text{temp}}$ with each $\Psi_n$ being the orthogonal projection onto $V_{s_n}$, and $\Psi_{\text{temp}}$ is the orthogonal projection onto $V_{\text{temp}}$. For $\Psi_{\text{temp}}$ we use decay of tempered matrix coefficients.  We have $$|\langle \pi(k_1a_tk_2)\Psi_{\text{temp}},\Psi_{\text{temp}}\rangle |\ll te^{-t}\|\Psi_{\text{temp}}\|_{W^{2,2}}^2$$ from Lemma \ref{lem:tempered}. The norm $W^{2,2}$ is the second order Sobolev norm; it includes the $L^2$ norm of the function as well as the $L^2$ norms of all derivatives of orders 1 and 2. Thus we get \beq\label{eq:tempered}\int_{g\in G} f_T(g)\langle \pi(g)\Psi_{\text{temp}},\Psi_{\text{temp}}\rangle dg\ll
\|Y_{a'b'cba}\|_{L^1}\|\Psi_{\text{temp}}\|_{W^{2,2}}^2 T^2\log T\ll
\eta^{-10}T^2\log T.\eeq  
Since complementary series representations of $G$ are irreducible, $\langle\pi(g)\Psi_{s_1},\Psi_{s_2}\rangle$ vanishes whenever $s_1$ and $s_2$ are distinct. For the same reason $\langle\pi(g)\Psi_{s_1},\Psi_{\text{temp}}\rangle=0$. 

To get the simple version of the Main Theorem \eqref{eq:simplemain} we use the decomposition $\Psi=\Psi_0+\Psi^\perp$, where $\Psi_0$ is the projection onto $V_{s_0}$ and $\Psi^\perp$ is the projection onto its orthogonal complement $V_{s_0}^\perp$. In this setting we can use Lemma \ref{lem:decay} to get $$|\langle \pi(k_1a_tk_2)\Psi^\perp,\Psi^\perp\rangle |\ll_\eps e^{(s_1-2+\eps)t}\|\Psi^\perp\|_{W^{2,2}}^2.$$ Then we have $$\int_{g\in G} f_T(g)\langle \pi(g)\Psi^\perp,\Psi^\perp\rangle dg\ll
\|Y_{a'b'cba}\|_{L^1}\|\Psi^\perp\|_{W^{2,2}}^2 T^{2s_1+\eps}\ll
\eta^{-10}T^{2s_1+\eps},  $$ and $\Psi_0$ contributes to the main term.

\section{$K$-type decomposition.} It  remains to consider the non-tempered part $$\sum_n\int_{g\in G}f_T(g)\langle \pi(g)\Psi_n,\Psi_n\rangle dg.$$ Expand $\Psi_n$ in the orthonormal basis $\{v_{lj}\}$ for the corresponding representation: $$\Psi_n=\sum_{lj}\langle \Psi_n,v_{lj}\rangle v_{lj}.$$ Here we use $v_{lj}\colon \Gamma\quot G\to\C$ to denote a basis in the automorphic model; note that matrix coefficients can be computed in any model, while evaluations $v_{lj}(g)$, $g\in G$ must be done in the automorphic model. In $KA^+K$ coordinates as in \eqref{eq:coordinates} and \eqref{eq:haar} we get 
\begin{multline}\label{eq:discarding}4\pi \sum_{l,j,l',j'}\int_{\substack{{K,M\quot K}\\{A^+,\Gamma\quot G}}}Y_{a'a;b'bc}(k_1k_2)\chi_{|a_t|<T}\langle \Psi_n,v_{lj}\rangle \langle v_{l'j'},\Psi_n\rangle \\ v_{lj}(xa_t k_2)\overline{v_{l'j'}(xk_1^{-1})}\, dx\, da_t\, dk_2\, dk_1.\end{multline} Now $Y_{a'b'cba}$ and $v_{lj}$ have the same $K$-transformation properties on the right. Using orthogonality of matrix coefficients for $K$ we get $$4\pi \overline{Y_{ac}(e)}Y_{a'c}(e)\frac{\langle \Psi_n,v_{ab}\rangle \langle v_{a'b'},\Psi_n\rangle}{(2a+1)(2a'+1)}\int_{t=0}^{2\log T}\langle \pi(a_t)v_{ac},v_{a'c}\rangle da_t.$$ Finally using the fact that $$\langle \Psi_n,v_{ab}\rangle =v_{ab}(e)+O(|\nabla v_{ab}(e)|\eta)$$ we get \begin{multline} 4\pi \overline{Y_{ac}(e)}Y_{a'c}(e)\frac{\overline{v_{ab}(e)}v_{a'b'}(e)}{(2a+1)(2a'+1)}\int_{t=0}^{2\log T}\langle \pi(a_t)v_{ac},v_{a'c}\rangle da_t+\\+O(\eta|\nabla v_{ab}(e)|\overline{Y_{ac}(e)}Y_{a'c}(e)\frac{v_{a'b'}(e)}{(2a+1)(2a'+1)}\int_{t=0}^{2\log T}\langle \pi(a_t)v_{ac},v_{a'c}\rangle da_t)+\\+O(\eta|\nabla v_{a'b'}(e)|\overline{Y_{ac}(e)}Y_{a'c}(e)\frac{\overline{v_{ab}(e)}}{(2a+1)(2a'+1)}\int_{t=0}^{2\log T}\langle \pi(a_t)v_{ac},v_{a'c}\rangle da_t)+O(\eta). \label{eq:mainterms}\end{multline} The error depends on Lipschitz norms of $v_{ab}$ and $v_{a'b'}$; we will control their dependence on the indices in an indirect way later. 

\begin{remark}
A similar expression can be found in Proposition 3.13 of \cite{bourgain_sector_2010}, but the significance of the Lipschitz norm is overlooked. In Lemma \ref{lem:lipschitz} we show how to control these norms, and the same treatment will work in the case of \cite{bourgain_sector_2010}. 
\end{remark}

\section{Main term.} In the computation of the main terms of the Theorem we deliberately avoid most complications by working out a skeletal form first and then filling in the constants as necessary. 
\begin{prop}\label{lem:skeletal} We have \begin{multline}\sum_{\substack{{\gamma\in\Gamma}\\{|\gamma|<T}}}Y_{a'b'}(k_1(\gamma))\overline{Y_{ab}(k_2^{-1}(\gamma))}=c_0T^{2\delta} +c_1T^{2s_1}+\dots+c_dT^{2s_d}+\\+O\left(T^{2\frac{10\delta+1}{11}}(\log T)^{1/11}(a+1)^{15/11}(a'+1)^{15/11}\right)+O(\dots)\end{multline} for certain constants $c_n=c_n(a,b,a',b'),$ $0\le n\le d$; three dots stand for the error terms of equation \eqref{eq:mainterms} with $\eta$ as in equation \eqref{eq:eta} below. For $c\ne 0$ we have  \begin{equation}\sum_{\substack{{\gamma\in\Gamma}\\{|\gamma|<T}}}Y_{a'a;b'bc}(K(\gamma))=O\left(T^{2\frac{10\delta+1}{11}}(\log T)^{1/11}(a+1)^{15/11}(a'+1)^{15/11}\right).\end{equation} 
\end{prop}

\begin{lemma}
With notation as before the asymptotics for matrix coefficients of $(\pi_s,V_s)$ are given by $$\langle \pi(a_t)v_{ac}^s,v_{a'c}^s\rangle=\mbox{\const} (e^{t(s-2)-|c|t}+O((a^2+a'^2+1)e^{-st-|c|t})).$$
\end{lemma}

\begin{proof}
Without loss of generality take $c\ge 0$. From Lemma \eqref{lem:generalform} we have  $$v_{ac}(r,\alpha)= P_{2a-c,c}(r)(1+r^2)^{-s-a}$$ for a polynomial $P_{2a-c,c}$. The powers appearing in the polynomial have the same parity; the lowest and the highest powers are $c$ and $2a-c$, respectively. Combining $P_{2a-c,c}(re^t)$ and $P_{2a'-c,c}(r)$ we can write the matrix coefficient as $$e^{ts}\int_{r=0}^\infty \frac{Q_{a+a'-c,c}(r^2;e^t)}{{(r^2e^{2t}+1)^a(r^2+1)^{a'}}}\left(\frac{r^2+1}{r^2e^{2t}+1}\right)^{s-2}\frac{rdr}{(r^2e^{2t}+1)^2}.$$ $Q$ is a polynomial with indicated maximal and minimal powers of $r^2$, and its coefficients depend on $e^t$ as inherited from the product $P_{2a-c,c}(re^t)\cdot P_{2a'-c,c}(r)$. It is natural to substitute $u=r^2$, followed by $$w=\frac{u+1}{ue^{2t}+1}.$$ These give \begin{multline}e^{ts}\int_{w=1}^{e^{-2t}}  \frac{Q_{a+a'-c,c}\left(\frac{1-w}{we^{2t}-1};e^t\right)}{\left(\frac{1-e^{-2t}}{w-e^{-2t}}\right)^a\left(\frac{e^{2t}-1}{we^{2t}-1}\right)^{a'} w^{a'}} w^{s-2}\frac{dw}{1-e^{2t}}=\\=\frac{e^{ts}}{(e^{2t}-1)^{a+a'+1}}\int_1^{e^{-2t}}w^{s-2-a'} Q_{a+a'-c,c}\left(\frac{1-w}{we^{2t}-1};e^t\right)(we^{2t}-1)^{a+a'}dw\end{multline} ignoring absolute constants. Now the integrand is a finite sum of powers of $w$ and can be easily evaluated once $Q$ is given. In particular it follows that the main term is of the form $\const\cdot e^{ts-2t-ct}$ and that the next highest term is $O((a^2+a'^2+1)e^{-st-ct})$ from \ref{lem:generalform}, as desired.
\end{proof}

\begin{proof}[Proof of Proposition \ref{lem:skeletal}]
It is enough to extract the correct power of $T$ from  the integral $$\int_{t=0}^{2\log T}\langle \pi(a_t)v_{ac},v_{a'c}\rangle da_t$$ from \eqref{eq:mainterms} and to collect the error terms. Using the previous Lemma and the Haar measure on $A^+$ from \eqref{eq:haar_a} we have that $$\int_{t=0}^{2\log T}\langle \pi(a_t)v_{ac},v_{a'c}\rangle \sh^2 t\, dt=\int_{t=0}^{2\log T}e^{ts-|c|t}dt + O(T^{2(2-s)}).$$
Thus we get 
\beq\label{eq:leading}\sum_{n=0}^d c_n T^{2s_n}\eeq for $c$ different from $0$ as in the statement. 

The error term in the Theorem comes from the error contributions \eqref{eq:smoothing}, \eqref{eq:mainterms}, and \eqref{eq:tempered}. 

The optimal choice for $\eta$ that makes the errors from \eqref{eq:smoothing} and \eqref{eq:tempered} equal: \beq\label{eq:eta}\eta=T^{\frac2{11}(1-\delta)}(\log T)^{1/11}(a+1)^{-\frac3{22}}(a'+1)^{-\frac3{22}}.\eeq This gives the advertised error term in the statement. 
\end{proof}

%For $n\ne 0$ it is clear that we get \eqref{eq:leading} as the main term, but understanding $c_n$ is quite hard. Therefore we only attempt to understand $c_0$, which is built from the Patterson-Sullivan measure. It is not too hard to carry out this calculation in the cases $$a=b=a'=b'=0,$$ $$a'=b'=0, a=b=1,$$ $$a=b=a'=b'=1.$$ The matrix coefficient is to be computed in the line model. The first case is easy, and it gives the correct main term in the Theorem. For the second and third case we need the value of $v_{11}(e)$. It is obtained by using the operator $J^+$ on $v_{00}$ with appropriate normalization \eqref{eq:Jnorm}. The matrix coefficients are computed using the formulae for $v_{10}$ and $v_{00}$ presented above. 

We outline the calculation of the main term for the case $a=b=a'=b'=0$. Plugging these values into \eqref{eq:mainterms} gives $$4\pi {\overline{v_{00}(e)}v_{00}(e)}\int_{t=0}^{2\log T}\langle \pi(a_t)v_{00},v_{00}\rangle da_t.$$  We have $$v_{00}(e)=\hat\nu(0,0).$$ For the matrix coefficient we have $$\langle\pi(a_t) v_{00},v_{00}\rangle=\frac{\sh(t\delta-t)}{(\delta-1)\sh t},$$ which is computed in the line model using Lemma \ref{lem:line_normalization}. The leading term (for large $t$) of this expression is $$\frac{e^{t\delta-2t}}{\delta-1}.$$ Using the measure from \eqref{eq:haar_a} for the integral we get the leading term of $$\frac{\pi}{\delta(\delta-1)}\hat\nu(0,0)\overline{\hat\nu(0,0)} T^{2\delta},$$ as needed. 

\section{Conjugation invariance.} Generalizing this computation seems very computationally intensive, so we use a trick to show that once the Theorem holds for the case of trivial spherical harmonics we have just considered, it holds in all cases. For $g\in G$, the spectral theory of $\Gamma\quot G$ and $g^{-1}\Gamma g\quot G$ is the same. So we consider the discrete group $g^{-1}\Gamma g$ and ``linearize'' in $g$. In essence we need to prove a multiplicity one statement for the coefficient of the main term by exploiting the transformation property under conjugation. We proceed inductively. The base case (constant spherical harmonics) was outlined above. The induction step is built on the next proposition. 

Write $\inv\colon G\to G$ for the inverse map. For functions $F_1, F_2, F_3$ on $K/M$, $A^+$, and $M\quot K$, respectively, let $$F_1\otimes F_2\otimes F_3(g)=F_1(k_1(g))F_2(a(g))F_3(k_2(g)).$$

\begin{prop}\label{prop:jumping}
Let $Y_{aa}$ and $Y_{a'a'}$ be a pair spherical harmonics, and suppose by induction that 
$$\lim_{T\to\infty}\frac1{T^{2\delta}}\sum_{\gamma\in\Gamma}Y_{a'a'}\otimes\chi_T\otimes \overline{Y_{aa}\circ\inv}(\gamma)=\frac{\pi}{\delta(\delta-1)}\hat\nu_\Gamma(a',a')\overline{\hat\nu_\Gamma(a,a)}.$$ Then \begin{multline}\lim_{T\to\infty}\frac1{T^{2\delta}}\sum_{\gamma\in\Gamma}\left(Y_{a'a'}\otimes\chi_T\otimes \overline{Y_{a+1,a+1}\circ\inv}(\gamma)+Y_{a'+1,a'+1}\otimes\chi_T\otimes \overline{Y_{aa}\circ\inv}(\gamma)\right)=\\=\frac{\pi}{\delta(\delta-1)}\left(\hat\nu_\Gamma(a',a')\overline{\hat\nu_\Gamma(a+1,a+1)}+\hat\nu_\Gamma(a'+1,a'+1)\overline{\hat\nu_\Gamma(a,a)}\right).\end{multline}
\end{prop}
% $$\sum_{\gamma\in\Gamma}Y_{a'b'}\otimes \chi_{|g|<T}\otimes \overline{Y_{ab}}(g\gamma g^{-1})=\frac{4\pi^2}{\delta(\delta-1)}\int Y

%We have $$\lim_{T\to\infty}\frac1{T^\delta}\sum_{\gamma\in\Gamma,|\gamma|<T}Y_{11}(k_1^{-1}(\gamma))\overline{Y_{11}(k_2(\gamma))}=\frac{4\pi^2}{\delta(\delta-1)} \hat\nu(1,1)\overline{\hat\nu(1,1)}.$$
%Applying this to the group $g\Gamma g^{-1}=\Gamma_g$ we have \beq\label{eq:conjugate}\lim_{T\to\infty}\frac1{T^\delta}\sum_{\gamma\in\Gamma,|g^{-1}\gamma g|<T}Y_{11}(k_1^{-1}(g^{-1}\gamma g))\overline{Y_{11}(k_2(g^{-1}\gamma g))}=\frac{4\pi^2}{\delta(\delta-1)} \hat\nu_{\Gamma_g}(1,1)\overline{\hat\nu_{\Gamma_g}(1,1)}.\eeq For $g$ close to the identity $|g^{-1}\gamma g|\approx |\gamma|$. 

\begin{remark}In effect we suppose that the main term of Theorem \ref{th:main} for one pair of spherical harmonics and for every group $\Gamma$, and prove that the main term is right for a new combination of pairs of spherical harmonics. 
\end{remark}

We need several lemmas to prove this Proposition.

\begin{lemma}
Let $\nu_{\Gamma}(z)$ be the Patterson-Sullivan measure for $\Gamma$ and let $g\in G$. Then, the Patterson-Sullivan measure for the conjugate of $\Gamma$ is $$\nu_{g^{-1}\Gamma g}(z)=P^\delta(gj,gz)\nu_{\Gamma}(gz). $$
\end{lemma}

\begin{proof}
This is immediate from the properties of the Patterson-Sullivan measure: the measure for the conjugated group is a combination of a push-forward with a Jacobian. 
\end{proof}

\begin{lemma}\label{lem:action_on_poisson}
We have $$\left.P^\delta (J^+j; \phi,\theta)\right|_{e\in G} = -\delta\sin\theta e^{i\phi}.$$ 
\end{lemma}

\begin{proof}
This is confirmed by direct calculation using \eqref{eq:formulae}--\eqref{eq:formulae_last} applied to the Poisson kernel \eqref{eq:poisson_kernel}  raised to the power $\delta$. 
\end{proof}

%\begin{lemma}\label{lem:toinfinity}
%Let $X\in\mathfrak g$. Let $f\colon G\to\C$ be a smooth function that does not depend on the $\psi$ variable in the $KA^+K$ decomposition \eqref{eq:coordinates}. Then we have  $$\lim_{t\to\infty} \pi(X).f(g)=f(-X(\theta_2,\phi_2)).$$ On the left $f$ is a function on the group, and $\pi$ is the right regular representation. On the right $f$ is a function on $\d\H^3$, and $-X$ acts by a M\"obius transformation. Convergence to the limit is exponential. 
%\end{lemma}

%Now take $g\in U(\mathfrak g^\C)$ examine the formulae \eqref{eq:formulae}--\eqref{eq:formulae_last} for elements of $\mathfrak g$ in our basis; of course we smooth the indicator in $t$ for this. The dependence on $t$ comes in two  flavors at infinity: either the coefficients decay, or become independent of $t$. We drop those that decay; they will not contribute to the limit. Since we only want the main term, we take the remaining coefficients to their limits are $t\to\infty$ and drop terms that contain  $\d_\psi$ since they don't contribute, giving  
%\begin{proof} Using \eqref{eq:formulae}--\eqref{eq:formulae_last} we compute the left hand side a basis to be 

%The right hand side is easily seen to match these formulae. 

%\end{proof}

\begin{proof}[Proof of Proposition \ref{prop:jumping}]
Approximate the indicator $\chi_T$ in the definition of $f_T$ as follows. Let $\Xi_{T,A}^+$ be 1 on $[0,T]$, 0 on $[T+A,\infty)$, and interpolate linearly on $[T,T+A]$. Similarly let $\Xi_{T,A}^-$ be 1 on $[0,T-A]$, 0 on $[T,\infty)$, and interpolate linearly on $[T-A,T].$ We will choose $A<T$ later. We consider $\Xi^+$ below; the treatment of $\Xi^-$ is the same. We have \begin{multline}\sum_{\gamma\in\Gamma} Y_{a'b'}\otimes\Xi_{T,A}^+\otimes \overline{Y_{ab}\circ\inv}(g^{-1}\gamma g)=\\=\frac{\pi}{\delta(\delta-1)}\hat\nu_{g^{-1}\Gamma g}(a',b')\overline{\hat\nu_{g^{-1}\Gamma g}(a,b)}T^{2\delta}+c_1 T^{2s_1}+\dots+c_d T^{2s_d}+\\+O(T^{2\frac{10\delta+1}{11}}(\log T)^{1/11}(a+1)^{15/11}(a'+1)^{15/11})+O(T^{\delta-1}A).\end{multline} Let $E>\eps>0$ and let $g=\exp \eps X$ for some $X\in \mathfrak g$. It is clear that all constants --- both implied and the constants $c_n$ --- are $1+O(E)$. 

Now we expand both sides in $\eps$. The left hand side reads \begin{multline*}\sum_{\gamma\in\Gamma} \left(Y_{a'b'}\otimes\Xi_{T,A}^+\otimes \overline{Y_{ab}\circ\inv}(\gamma)\right.+\\+ \eps Y_{a'b'}\otimes\Xi_{T,A}^+\otimes \overline{Y_{ab}\circ\inv}(\gamma X)+\eps Y_{a'b'}\otimes\Xi_{T,A}^+\otimes \overline{Y_{ab}\circ\inv}(-X\gamma )+\\\left.+O(\eps^2 1\otimes\Xi_{T,A}^+\otimes 1(\gamma))\right).\end{multline*} The two $O(\eps)$ terms are similar; we only treat the first one. As $t$ becomes large the only derivatives that contribute are 
\begin{align*}
\frac h2&=\cos \theta_2 \d_{t}- \sin \theta_2 \d_{\theta_2} \label{eq:formulae_simple}
\\
\frac{\Ih}2&=\d_{\phi_2}
\\
\frac e2&=\frac{1}{2} \cos \phi_2 \sin \theta_2 \d_{t}-\frac{1}{2} (\ctg \theta_2+\cosec \theta_2) \sin \phi_2 \d_{\phi_2}+\frac{1}{2} \cos \phi_2 (1+\cos \theta_2 ) \d_{\theta_2}
\\
\frac{\Ie}2&=-\frac{1}{2} \sin \phi_2 \sin \theta_2 \d_{t}-\frac{1}{2} \cos \phi_2 (\ctg \theta_2+ \cosec \theta_2) \d_{\phi_2}-\frac{1}{2} (1+\cos \theta_2) \sin \phi_2 \d_{\theta_2}
\\
\frac f2&=\frac{1}{2} \cos \phi_2 \sin \theta_2 \d_{t}+\frac{1}{2} (\ctg \theta_2- \cosec \theta_2) \sin \phi_2 \d_{\phi_2}+\frac{1}{2} \cos \phi_2 (-1+\cos \theta_2 ) \d_{\theta_2}
\\
\frac{\If}2&=\frac{1}{2} \sin \phi_2 \sin \theta_2 \d_{t}-\frac{1}{2} \cos \phi_2 (\ctg \theta_2- \cosec \theta_2) \d_{\phi_2}+\frac{1}{2} (-1+\cos \theta_2) \sin \phi_2 \d_{\theta_2}
\end{align*}
It is enough to consider the case $X=J^+$. Here we have \beq\label{eq:Jminus}J^+=e^{i\phi_2}\sin\theta_2\d_t+i (\ctg \theta_2- \cosec \theta_2) e^{i\phi_2} \d_{\phi_2}+e^{i\phi_2} (-1+\cos \theta_2 ) \d_{\theta_2}.\eeq 
Suppose $a=b$ and $a'=b'$. The effect of applying $J^+$ is as follows. The first term turns the function \beq\label{eq:function}Y_{a'a'}\otimes\Xi_{T,A}^+\otimes \overline{Y_{aa}\circ\inv}(\gamma)\eeq into $$-\frac{T}{2A}
Y_{a'a'}\otimes(\chi_{T+A}-\chi_T)\otimes \overline{Z_{a+1,a+1}  \circ\inv}(\gamma),$$ where $Z_{a+1,a+1}(\phi,\theta)=Y_{aa}(\phi,\theta) \sin\theta e^{i\phi}=\const\cdot Y_{a+1,a+1}(\phi,\theta).$ From  \eqref{eq:mainterms} and Lemma \ref{lem:skeletal} it follows that $$-\frac{T}{2A}\sum_{\gamma\in\Gamma}
Y_{a'a'}\otimes(\chi_{T+A}-\chi_T)\otimes \overline{Z_{a+1,a+1}  \circ\inv }(\gamma)$$ and $$-\delta\sum_{\gamma\in\Gamma}Y_{a'a'}\otimes\chi_T\otimes \overline{Z_{a+1,a+1}\circ\inv  }(\gamma)$$ have identical main terms (when $A\ll T$), so we use the second form. The other two terms of \eqref{eq:Jminus} turn the function \eqref{eq:function} into $$-aY_{a'a'}\otimes\chi_T\otimes \overline{Z_{a+1,a+1}  \circ\inv }.$$ The net contribution from $J^+$ is $$-(a+\delta)Y_{a'a'}\otimes\chi_T\otimes \overline{Z_{a+1,a+1}\circ\inv}.$$

Now we compare to the right hand side. After linearizing in $\eps$ we need to study two terms: $$\overline{\int Y_{aa}(\phi,\theta)P^\delta (J^+  j;\phi,\theta) d\nu_\Gamma(\phi,\theta)}+\overline{\int Y_{aa}(-J^+(\phi,\theta)) d\nu_\Gamma(\phi,\theta)}.$$ The action $-J^+(\phi,\theta)$ is by a M\"obius transformation on the boundary. Using Lemma \ref{lem:action_on_poisson} we rewrite the first term as $$-\delta\overline{\int Z_{a+1,a+1}(\phi,\theta)d\nu_\Gamma(\phi,\theta)}.$$  Carrying out the calculation gives $$-a\overline{\int Z_{a+1,a+1} (\phi,\theta)d\nu_\Gamma(\phi,\theta)}$$ for the second term. This matches the terms on the left-hand side. The term $$Y_{a'a'}\otimes\Xi_{T,A}^+\otimes \overline{Y_{aa}\circ\inv}(-X\gamma)$$ is treated in a similar fashion. 

Since $A$ and $\eps$ were arbitrary, $O(\eps)$ terms on both sides must match, whence the result. 
\end{proof}

\section{$K$ invariance.} The previous observation does not allow using completely general combinations of spherical harmonics due to symmetry, so we use $K$ invariance on one side to break this symmetry. 

\begin{lemma}\label{lem:kinv}
Suppose have know that for a finite set $\mathcal I$ \begin{multline}\label{eq:kinv}\lim_{T\to\infty}\frac1{T^{2\delta}}\sum_{\gamma\in\Gamma,|\gamma|<T}\sum_{(L,J,L',J')\in \mathcal I}Y_{LJ}(k_1(\gamma))\overline{Y_{L'J'}(k_2^{-1}(\gamma))}=\\=\frac{\pi}{\delta(\delta-1)}\sum_{(L,J,L',J')\in \mathcal I}\hat\nu(L,J)\overline{\hat\nu(L',J')}.\end{multline} Then in fact $$\lim_{T\to\infty}\frac1{T^{2\delta}}\sum_{\gamma\in\Gamma,|\gamma|<T}Y_{LJ}(k_1(\gamma))\overline{Y_{L'J'}(k_2^{-1}(\gamma))}=\frac{\pi}{\delta(\delta-1)}\hat\nu(L,J)\overline{\hat\nu(L',J')}$$ for each $(L,J,L',J')\in \mathcal I$.
\end{lemma}

\begin{proof}
Redefine $f_T(g)$ from equation \eqref{eq:ftg} by $$f_T(g)=\sum_{(L,J,L',J')\in \mathcal I}Y_{LJ}(k_1(g))\chi_{|g|<T}\overline{Y_{L'J'}(k_2^{-1}(g))}$$ and let $\kappa\in K$. Consider the Theorem for the function $f_T(g\kappa).$ We can carry out the proof as before until we get to \eqref{eq:discarding}. This equation will now read \begin{multline}4\pi\sum_{\substack{{L,J,L',J',J''}\\{l,j,l',j'}}}\int_{\substack{{K,M\quot K}\\{A^+,\Gamma\quot G}}}Y_{LJ}(k_1^{-1})\overline{Y_{L'J''}(k_2)\rho_{J''J'}^{(L')}(\kappa)}\chi_{|a_t|<T}\\\langle \Psi_n,v_{lj}\rangle \langle v_{l'j'},\Psi_n\rangle v_{lj}(xa_t k_2)\overline{v_{l'j'}(xk_1^{-1})}\, dx\, da_t\, dk_2\, dk_1\end{multline} where $\rho$ is the corresponding representation of $K$. That is, each term in the sum  transforms under $ K$ in the same way as $$\frac{\pi}{\delta(\delta-1)}\sum_{(L,J,L',J')\in \mathcal I}\hat\nu(L,J)\overline{\hat\nu(L',J')}$$ transforms under $K.$ Repeating this approach on the other side ensures that \eqref{eq:kinv} holds term by term. 
\end{proof}

It is easy to see from Proposition \ref{prop:jumping} and Lemma \ref{lem:kinv} that the action of $U(\mathfrak g)$ is rich enough to produce every pair of spherical harmonics from the one we checked by hand originally, whence the main terms of Theorem \ref{th:main} are verified for all pairs of indices. 

%\subsection{Error term.} 

\section{Bounds in the variables $a$ and $a'$.\label{sec:lipschitz}}

We confirmed that $c_0(a,b,a',b')=\frac{\pi}{\delta(\delta-1)}\nu(Y_{a'b'})\overline{\nu(Y_{ab})}.$ By applying the same differential operators $R$, $L$, $J^\pm$ to the Patterson-Sullivan distributions \eqref{eq:distributions} we get that $$c_n(a,b,a',b')=\frac{\pi}{\delta(\delta-1)}D_{\Gamma,n}(Y_{a'b'})\overline{D_{\Gamma,n}(Y_{ab})}.$$ 

Now the bound on the coefficients \eqref{eq:bound_on_coeffs} in Theorem \ref{th:main} follows from the regularity property of $D_{\Gamma,n}$ in Section \ref{sec:ps} since $\|Y_{ab}\|_{L^\infty}\ll \sqrt{a+1}.$ 

\begin{remark}A bound in terms of $a, a'$ is also needed in \cite{bourgain_sector_2010}, where Theorem 1.5 includes the incorrect claim that $c_n(a,a')$ are bounded for each $n$. Nevertheless it can be patched into a true statement $$|c_n(a,a')|\ll (|a|+1)^{1-s_n}(|a'|+1)^{1-s_n}.$$ For $\SL(2,\R)$ regularity of automorphic distributions $D_{\Gamma,n}$ which we use here was analyzed by Schmid \cite{schmid_automorphic_sl2r} for all representations and by Otal \cite{otal_fonctions_propres} for the complementary series representations. The results of \cite{bourgain_sector_2010} are little affected by this patch: it can only offset the error term slightly for bisector counts involving non-smooth functions.  
\end{remark}

Finally we need to establish a bound for $\nabla v_{ab}(e)$ that was needed in Proposition \ref{lem:skeletal}. We rely on Theorem \ref{th:main}, which has been established except for the power of $a+1$, $a'+1$ in the error term; currently the error term depends on $\nabla v_{ab}(e)$. The main term has been confirmed, so we use it to prove the next 

\begin{lemma}\label{lem:lipschitz}
We have $\nabla v^s_{ab}(e)\ll (a+1)^{2-s}$ for $s=\delta$ and $\nabla v^s_{ab}(e)\ll (a+1)^{4-2s}$ for $s\ne \delta.$
\end{lemma}

\begin{proof}
The action of $\mathfrak g$ on $v_{ab}$ is characterized by equations \eqref{eq:Rnormalization}, \eqref{eq:Lnormalization}, \eqref{eq:haction}, \eqref{eq:Jmaction}, \eqref{eq:Jpaction} above. It follows in particular that for any $X\in\mathfrak g$ of norm 1 we have $$\pi(X)v_{ab}=\sum_{i=-1}^1\sum_{j=-1}^1 C_{ij} v_{a+i,b+j}$$ with each $C_{ij}\ll a+1$. Therefore \beq\label{eq:lipschitz_control}\pi(X)v_{ab}(e)\ll (a+1)\sum_{i=-1}^1\sum_{j=-1}^1 |v_{a+i,b+j}(e)|. \eeq From the fact that the main term of Theorem \ref{th:main} is correct we glean that $$D_{\Gamma,n}(Y_{ab})\sim\const\cdot v_{ab}(e) (a+1)^{s-1/2}$$ for some absolute constant. The result follows. 
\end{proof}

This Lemma provides the control needed in Proposition \ref{lem:skeletal}, showing that the two error terms there are of the same order in $a$ and $a'$. The proof of Theorem \ref{th:main} in thus complete.

%Finally we show that the error terms in \ref{lem:skeletal} are of the same order.
%Since $\|Y_{a'b'cba}\|_{L^\infty}\ll\sqrt{ (a+1)(a'+1)},$ we have $$\left|\sum_{\substack{{\gamma\in\Gamma}\\{|\gamma|<T}}} Y_{a'b'cba}(K(\gamma))\right|\ll\sqrt{(a+1)(a'+1)} \sum_{\substack{{\gamma\in\Gamma}\\{|\gamma|<T}}} 1.$$ Plugging in the formulae we have obtained for both expressions gives 
%\begin{multline}\left|\sum_{n=0}^dc_n(a,b,a',b')T^{2s_n}+O(T^{2\frac{10\delta+1}{11}}((a+1)(a'+1))^{\frac{15}{11}})+O_{a,b,a',b'}(T^{2\frac{10\delta+1}{11}}(\log T)^{\frac1{11}})\right|\ll\\\ll\sqrt{(a+1)(a'+1)}  \left(\sum_{n=0}^dc_n(0,0,0,0)T^{2s_n}+O(T^{2\frac{10\delta+1}{11}}(\log T)^{1/11})\right).\label{eq:lipschitz_error}\end{multline} We already know what $c_n$ are, but we don't write them out. Now fix $T$; consider \eqref{eq:lipschitz_error} as a statement about functions of $a$ and $a'$ only. It reads $$O(\sqrt{(a+1)(a'+1)})+O(((a+1)(a'+1))^{15/11})+O_{a,a'}(1)\ll O(\sqrt{(a+1)(a'+1)}).$$ Therefore $$O_{a,a'}(1)=O(((a+1)(a'+1))^{15/11}),$$ so that the mysterious error term with unspecified dependence on $a$, $a'$ can be incorporated into $O(((a+1)(a'+1))^{15/11})$. This completes the proof of Theorem \ref{th:main}. 

%% file: proof_of_th_apollonian.tex
\chapter{Proof of Theorems \ref{th:apollonian} and \ref{th:general_count}}

%Using the effective estimate for spherical harmonics we get an error term for bisectors. Let $S_1$ and $S_2$ be nice sectors in $K/M$ and $M\quot K$ resp. Nice means that their boundary has PS measure zero and that they are bounded by two longitude lines and two latitude lines. We never need to consider sectors in all of $K$ since $$\sum_{\substack{{\gamma\in\Gamma}\\{|\gamma|<T}}} f_1((\phi,\theta)^{-1}) e^{im \psi} f_2(\theta_2,\phi_2)=O(T^{1+\eps})$$ if $m\ne 0$. Here we think of the usual $K_1.A^+.M\quot K_2$ decomposition coordinatized as before by Euler angles $\phi$, $\theta$, $\psi$ for $K_1$,   $t$ for $A^+$, and $\theta_2$, $\phi_2$ for $M\quot K_2.$ Then an approximation argument shows that $$\sum_{\substack{{\gamma\in\Gamma}\\{|\gamma|<T}}}\chi_{S_1}(k_1^{-1})\chi_{S_2}(k_2)=\frac{4\pi^2 T^\delta}{ \delta(\delta-1)}\nu(S_1)\nu(S_2)+O(T^\alpha)+O(T^{s_1})$$ where $$\alpha=\delta\left(1-\frac{\delta-1}{11(\delta+2(75/22+\eps))}\right)\approx 0.996579\delta+\eps.$$ Here $s_1$ is the second smallest eigenvalue, which we don't know. 

\section{Setup.} Consider the sum \beq \label{eq:counting_apollonian}S_{v,\Gamma}(T)=\sum_{\substack{{\|\gamma v\|<T}\\{\gamma\in\Gamma}}}\chi_{G/(\Stab v\cap \Gamma)}(\gamma)\eeq with $v\ne 0$ a column vector in $\R^4$ in the cone $\{Q(v)=0\}$. Here $Q$ is a quadratic form of signature $(3,1)$, and assume that an isomorphism $\vartheta\colon\PSL(2,\C)\to\SO_Q^\circ(\R)$ has been fixed. The norm $\|\cdot\|$ on $\R^4$ can be any norm whatsoever; in fact, it need not even satisfy the triangle inequality, but must have the scaling property, some non-vanishing, and continuity. For the Apollonian circle packing problem the form of interest is the Descartes form $$Q_D(a,b,c,d)=a^2+b^2+c^2+d^2-\frac12 (a+b+c+d)^2.$$ One can check that \beq \label{eq:isomorphism}g\mapsto qgq^{-1}\eeq for $$q=\frac12\begin{pmatrix} 1&-1&-1&1\\-1&1&-1&1\\-1&-1&1&1\\1&1&1&1\\\end{pmatrix}$$ is an isomorphism $\SO(x^2+y^2+z^2-w^2)\cong \SO_{Q_D}(\R)$. This map composed with $\iota$ from \eqref{eq:iota} defines the action of $\PSL(2,\C)$ on $\R^4$. Let $\tilde\Gamma$ be   the Apollonian group \cite{kontorovich_apollonian_2011, lagarias1}. It is generated by the matrices $S_1,$ $S_2$, $S_3$, $S_4$ given by \beq\label{eq:generators}\begin{pmatrix}-1&2&2&2\\&1\\&&1\\&&&1\end{pmatrix},\quad \begin{pmatrix}1\\2&-1&2&2\\&&1\\&&&1\end{pmatrix}, \quad \begin{pmatrix}1\\&1\\2&2&-1&2\\&&&1\end{pmatrix}, \quad \begin{pmatrix}1\\&1\\&&1\\2&2&2&-1\end{pmatrix}\eeq inside $\O_{Q_D}(\R).$ Let $\Gamma$ be the image of $\tilde\Gamma$ under the projection $\O_{Q_D}(\R)\to\SO_{Q_D}(\R).$ Observe that these generators have determinant $-1$, so that products of pairs of generators are in $\SO_{Q_D}(\R)$. Thus we can count even length words in $\tilde\Gamma$ directly, and count odd length words by looking at the action of $\tilde\Gamma$ on $S_1v$. In what follows we consider only one of these sums as the treatment of the other sum is identical. 
%The group elements $k_1$, $a_t$, $k_2$ need to be conjugated to $\SO(Q_D)$ accordingly. 

\section{Region of summation.} The set $G/(\Stab v\cap \Gamma)$ should be interpreted as one fundamental domain under the action of $\Stab v\cap \Gamma$ on $G$. It is known that $\Stab v\cong NM$ since $v$ lies in the cone $\{Q=0\}$. Since the action of $\SO_Q(\R)$ on the cone is transitive, there are $u\in\R^4$ and $g\in G$ so that $v=gu$ and $\Stab u=NM$. Then we can rewrite the sum as $$S_{v,\Beta}(T)=\sum_{\|g\beta u\|<T,\beta\in \Beta}\chi_{G/(\Stab u\cap \Beta)}(\beta)$$ with $\Beta=g^{-1}\Gamma g$. Now $\Stab u=NM$, and by an argument from \cite{kontorovich_apollonian_2011} $NM\cap\Beta=N\cap\Beta$. Therefore $\Stab u\cap \Beta$ is a discrete subgroup of $\R^2$ and can be isomorphic to the trivial group, $\Z$, or $\Z^2$. In the case of Apollonian circle packings only two of these options are possible, the trivial group and $\Z$. The first is realized by a bounded packing such as the one in Figure \ref{fig:boundedpacking}; the second corresponds to a periodic packing like the one in Figure \ref{fig:unboundedpacking}. 

Suppose the set $R=\Stab u\cap \Beta$ is non-trivial. If $G$ is written in Iwasawa coordinates $KAN$, then $R$ imposes a restriction only in the $N$ variable. That is, $R=K. A.\Proj_N(R)=K.\Proj_{AN}(R).$ We denote these projections as $R$, too. The region is a strip bounded by a pair of parallel lines or the intersection of two such strips. In the upper half-space model $AN\cong \H^3$ this lifts to a region between two parallel vertical planes or the intersection of two such regions. Suppose without loss of generality that $j\in \H^3$ is within this region. Let $R_1\subset \H^3$ be the region below geodesics joining $j$ to the boundary of $\Proj_N(R)\subset\d\H^3$. Clearly $R_1\subset R$; let $R_2$ be the complement of $R_1$ in $R$. It is easy to see that the region $R_2$ is contained within one fundamental domain for $\Gamma\subset G$. Therefore $R_2$ contains at most one point of $\Gamma$, whence $$S_{v,\Beta}(T)=\sum_{\substack{{\|g\beta u\|<T}\\{\beta\in\Beta}}}\chi_{R_1}(\beta)+O(1).$$ 
The region $R_1$ conveniently replaces $R$ and is easy to parametrize in $KA^+K$ coordinates. It is left $K$-invariant, and in the right $K$ factor it is supported on  directions emanating from $j$ which hit the boundary inside $\Proj_N(R)$. Thus $\chi_{R_1}(\beta)=\chi_{\Proj_{\d \H^3} R}(k_2(\beta)).$

Now in the $KA^+K$ decomposition we compute that $$g.k_1a_tk_2.u=\frac12(1+\cos\theta(k_2))|a_t|.gk_1u+O(1).$$ The main term in this expression vanishes (i.e., $\cos\theta(k_2)\ne-1$) only if $k_2u$ is orthogonal to the highest weight vector in the $K$-representation. Write $\gamma=k_1a_tk_2$ for $\gamma\in\Gamma$.  Then $\cos\theta(k_2(\gamma))=-1$ for only finitely many $\gamma\in\Gamma$ since its action is properly discontinuous; therefore we tacitly assume $\cos\theta(k_2(\gamma))\ne-1$.

%Let us only consider the case $v=v_0=(0,0,1,1)$ shown in Figure \ref{packing}; the others are obtained by conjugation. In the packing corresponding to $v_0$  we count circles between two vertical lines. Say we take the domain $\{|x_1|\le \frac12\}.$ Inside this set consider the region $S'$ below geodesics joining a point on the locus $\{|x_1|=\frac12\}\cap \{y=0\}$ to the point $j\in \H^3.$ The region above these geodesics contains but a finite number of points of $\Gamma$, so it will not affect the count for large $T$. Let $S$ be the image of $S'$ on the boundary (pushed from away from $j$). This allows us to think of $S$ as a subset of $K/M$. 
%Thus we consider $$\sum_{\substack{{\|v_0\gamma\|<T}\\{\gamma\in\Gamma}}}\chi_{S}(k_1(\gamma)). $$ To be consistent with \cite{kontorovich_apollonian_2011}, we switch to $a_t=\begin{pmatrix}e^{-t/2}\\&e^{t/2}\end{pmatrix},$ which also changes $A^+$ in our notation to $A^-$. It is clear that Theorem \ref{th:main} is true in this case, too. We compute that $$(q\cdot(k_1a_tk_2)\cdot q^{-1})\cdot v_0^T= \frac12 (\cos\theta_1-1) \cdot |a_t|\cdot   %\begin{pmatrix}1+\cos\theta_2+(\cos\phi_2+\sin\phi_2)\sin\theta_2\\1+\cos\theta_2-(\cos\phi_2+\sin\phi_2)\sin\theta_2\\1-\cos\theta_2+(\cos\phi_2-\sin\phi_2)\sin\theta_2\\1-\cos\theta_2+(-\cos\phi_2+\sin\phi_2)\sin\theta_2\end{pmatrix}^T
% k_2^{-1}v_0^T+O(1).$$ As before variables correspond to Euler angles in $k_1$, $k_2$. 

%(1+\cos\theta_2+(\cos\phi_2+\sin\phi_2)\sin\theta_2)

\section{Smoothing.} To prepare the sum for an application of Theorem \ref{th:main} we smooth the indicator $\chi_{R_1}$ to approximate from above and from below; let the smooth version $\chi_{R_1}^U$ be non-constant in a neighborhood of size $U$ about the boundary of the region $S$. This introduces an error of $\ll U^\delta T^\delta$. Also take a partition of unity $\{\rho_i^V,i\in I\}$ subordinate to an open cover $\{V_i\mid i\in I\}$ with $\diam(V_i)\le V$ for all $i\in I$; we can take $|I|\sim V^{-2}$. Using an obvious approximation argument we recast the sum as $$\sum_{i,j\in I} \sum_{\substack{{\beta\in\Beta}\\{|\beta|<\frac{T}{\frac12 |1+\cos\theta(k_2(\beta))| \|gk_1(\beta)u\|}}}}\chi_{R_1}^U(k_2(\beta))\rho_i^V(k_1(\beta))\rho_j^V(k_2(\beta)).$$

\section{Applying Main Theorem.} Fix $$W=\max_{\supp \nu\cap R_1}f,$$ which exists since the intersection is compact. The function $$f(k_2)=\frac1{\frac12(1+\cos\theta(k_2(\beta)))}$$ is unbounded, so we replace it by $$f(k_2)\wedge W$$ and smooth it at the edge; call the resulting smooth function $\tilde f(k_2)$. First we consider the bounded function $\tilde f$ and then estimate the error committed in passing from $f$ to $\tilde f$. 

Pick points $v_i\in V_i$ for each $i\in I$. Approximate the above sum by $$\sum_{i,j\in I} \sum_{\substack{{\beta\in\Beta}\\{|\beta|<\frac{T\tilde f(v_j)}{ \|gv_iu\|}}}}\chi_{R_1}^U(k_2(\beta))\rho_i^V(k_1(\beta))\rho_j^V(k_2(\beta)).$$  The error incurred at this step is $\ll V T^\delta$.   

%The natural way to proceed is to decompose functions of $k_1$ and $k_2$ into Fourier series and apply Theorem \ref{th:main}. This approach can be modified slightly in light of Remark \ref{rem:trivial}, which implies that for spherical harmonics with large indices the trivial bound may be better than the one furnished by Theorem \ref{th:main}. Let $A<T$ be a parameter chosen later depending on $T$. When $(a+1)(a'+1)<A$, we use Theorem \ref{th:main}; otherwise we use the observation in Remark \ref{rem:trivial}. 

From Theorem \ref{th:main} in the simplified form this sum equals $$\sum_{i,j}\frac{\nu(\chi^U\rho_j^V)\nu(\rho_i^V)(\tilde f(k_2^{-1}))^\delta}{\|gv_iu\|^\delta} T^\delta+O(T^{\frac{10\delta+s_1}{11}}V^{-4-2(\frac{15}{11}+2)+\eps}).$$ The error term comes from estimating Fourier coefficients of $\chi$ and $\rho$ and ensuring that the sum over $a$, $a'$ converges. It is clear that the largest error contribution will come from the partitions of unity, so we can set $U=V$. Now the sum over $i$ and $j$ approximates the respective integral with error $\ll VT^\delta.$ Thus the leading term becomes 
$$\frac{\pi}{\delta(\delta-1)}   \left(\int_{\d \H^3}\chi_{R_1}(k_2)\left(\frac2{1-\cos\theta(k_2)}\right)^\delta d\nu_\Beta(k_2)  \int_{\d \H^3} \frac{d\nu_\Beta(k_1)}{\| g k_1u\|^\delta}   \right)T^\delta$$ as in the statement of the Theorem; here we have used the fact that $f$ and $\tilde f$ only differ outside of the support of $\nu$. In the more common coordinates on $\d \H^3=\C\cup\infty $ the first factor reads $$\int_{(\Beta\cap N)\quot \d \H^3}(|z|^2+1)^{\delta}d\nu_\Beta(z).$$ The error term comes from optimizing the two different error terms from this section. They are \beq VT^\delta,\, T^{\frac{10\delta+s_1}{11}} V^{-4-2(\frac{15}{11}+2)+\eps}.\label{eq:error_terms}\eeq Equating them gives $V=T^{-\frac{\delta-s_1}{129}}$, which works for all counting theorems stated above.  

Now we estimate the error from bounding $f$. Sacrificing a constant factor we can focus on $k_2$: \begin{multline}\!\sum_{i,j\in I}\Bigg(\sum_{|\beta|<\frac{Tf(k_2(\beta))}{ \|gk_1(\beta)u\|}}\chi_{R_1}^U(k_2(\beta))\rho_i^V(k_1(\beta))\rho_j^V(k_2(\beta)) -\\- \sum_{|\beta|<\frac{T\tilde f(k_2(\beta))}{ \|gk_1(\beta)u\|}}\chi_{R_1}^U(k_2(\beta))\rho_i^V(k_1(\beta))\rho_j^V(k_2(\beta))\Bigg)\ll \\\ll \sum_{i\in I}\sum_{|\beta|<T f(k_2(\beta))}\chi_{\{f> W\}}(k_2)\rho_i^V(k_2)\ll \sum_i \sum_{|\beta|<T f(k_i^\ast)}\chi_{\{f> W\}}(k_2)\rho_i^V(k_2) ,\end{multline}
where $k_i^\ast\in V_i$ is a point where $f$ attains its maximum over $V_i$. Applying the Main Theorem in the simplified form gives $$\ll \sum_i (f((k_i^\ast)^{-1})T)^{\frac{10\delta+s_1}{11}+\eps}\sum_a a^{\frac{15}{11}+1}(aV)^{-l}.$$ To ensure convergence in $a$ take $l>\frac{15}{11}+2$ and rewrite the expression as $$\left(V^2\sum_i (f((k_i^\ast)^{-1}))^{\frac{10\delta+s_1}{11}+\eps}\right)\cdot T^{\frac{10\delta+s_1}{11}+\eps} V^{-2-2-\frac{15}{11}}. $$ The quantity in parentheses asymptotically is a finite integral since $f$ as a simple pole. The resulting error term is smaller than both error contributions in \eqref{eq:error_terms}, so neglecting this term is justified.

\section{Ideal triangle case.}
Consider the function $N^P(T)$ for counting circles in an ideal triangle, not in the entire packing. Suppose that $v\in\R^4$ consists of the curvatures of the four largest circles in Figure \ref{fig:apollonian_packing_ideal}. A circle from the packing that is contained in a particular ideal triangle is characterized by the property that the quadruple in which it appears has prescribed form in terms of generators \eqref{eq:generators}. To wit, the first generator cannot be one that flips the smallest circle out of the triangle; say the circles are ordered so that this generator is $S_4$. The group $\{1,S_4\}$ acts on $G$ (and also on $K\quot G=\H^3$) from the right; let $G_4$ be the fundamental domain for this action that lies under the hemisphere fixed by $S_4$. That is, $G_4$ contains the ideal triangle in Figure \ref{fig:apollonian_packing_ideal} by implicitly identifying $G_4$ with its projection onto $\H^3$ and $\d\H^3$. The counting function $S_{v,\Gamma}(T)$ from \eqref{eq:counting_apollonian} should be changed to \beq \label{eq:counting_apollonian_triangle}S_{v,\Gamma}^{\Delta}(T)=\sum_{\substack{{\|\gamma v\|<T}\\{\gamma\in\Gamma}}}\chi_{G_4/(\Stab v\cap \Gamma)}(\gamma).\eeq 

Proceeding as in the case of full packings  we get the formula of Theorem \ref{th:ideal} with $u$ and $\Beta$ defined as before.